\documentclass[11pt]{amsart}
\usepackage[utf8]{inputenc}
\usepackage{graphicx} 

\usepackage{fullpage, amsmath, amssymb, latexsym, amsfonts, amscd, amsthm, extarrows, mathtools, tikz-cd}

\usepackage{upgreek}
\usepackage{comment}
\usepackage{enumitem}

\usepackage{hyperref}
\hypersetup{pdftoolbar=true, pdftitle={Template}, pdffitwindow=true, colorlinks=true, citecolor=blue, filecolor=black, linkcolor=black, urlcolor=red, hypertexnames=false}

\usepackage{xcolor}
\newcommand {\DM}[1] {{\color{red}\sf DM : [#1]}}

\newtheoremstyle{thmlike}
{8pt}
{3pt}
{\slshape}
{}
{\bfseries}
{.}
{1em}
{}

\newtheoremstyle{deflike}
{8pt}
{3pt}
{}
{}
{\bfseries}
{.}
{1em}
{}

\theoremstyle{thmlike}
\newtheorem{theorem}{Theorem}[section]

\newtheorem{lemma}[theorem]{Lemma}
\newtheorem{corollary}[theorem]{Corollary}

\theoremstyle{deflike}
\newtheorem{definition}[theorem]{Definition}
\newtheorem{remark}[theorem]{Remark}

\newtheorem{hypothesis}[theorem]{HYPOTHESIS}

\newcommand{\inv}{\mathrm{inv}}
\newcommand{\nil}{\mathrm{nil}}
\newcommand{\norm}{\mathrm{norm}}
\newcommand{\opp}{\mathrm{op}}
\newcommand{\Pord}{P-\mathrm{ord}}
\newcommand{\Pword}{P_w\mathrm{-ord}}
\newcommand{\Paord}{P-\mathrm{a.ord}}
\newcommand{\Pwaord}{P_w-\mathrm{a.ord}}

\newcommand{\sub}{\mathrm{sub}}
\newcommand{\tor}{\mathrm{tor}}
\newcommand{\tp}[1]{\prescript{t}{}{#1}} 
\newcommand{\level}[2]{\prescript{}{#1}{#2}} 

\newcommand{\diag}{\operatorname{diag}}

\newcommand{\End}{\operatorname{End}}

\renewcommand{\ker}{\operatorname{ker}}
\renewcommand{\hom}{\operatorname{Hom}}

\newcommand{\isom}{\operatorname{Isom}}

\newcommand{\Sh}{\operatorname{Sh}}

\newcommand{\Gm}{\operatorname{\mathbb{G}_{m}}}
\newcommand{\Ga}{\operatorname{\mathbb{G}_{a}}}
\newcommand{\Gal}{\operatorname{Gal}}
\newcommand{\GL}{\operatorname{GL}}

\newcommand{\Lie}{\operatorname{Lie}}
\newcommand{\SL}{\operatorname{SL}}

\renewcommand{\det}{\operatorname{det}}

\newcommand{\ord}{\operatorname{ord}}
\newcommand{\rk}{\operatorname{rank}}
\newcommand{\tr}{\operatorname{trace}}
\renewcommand{\v}{\operatorname{\nu}} 
\newcommand{\vol}{\operatorname{vol}}

\newcommand{\ind}[2]{\operatorname{\iota}_{#1}^{#2}}
\newcommand{\Ind}{\operatorname{Ind}}
\newcommand{\Rep}{\operatorname{Rep}}

\renewcommand{\r}[2]{\operatorname{r}_{#1}^{#2}}
\newcommand{\cg}[1]{\widetilde{#1}} 

\renewcommand{\AA}{\mathbb{A}}
\newcommand{\Ab}{\mathcal{A}}
\newcommand{\CC}{\mathbb{C}}
\newcommand{\EE}{\mathcal{E}}

\newcommand{\KK}{\mathcal{K}}
\newcommand{\LL}{\mathcal{L}}
\newcommand{\MM}{\mathrm{M}}
\newcommand{\OO}{\mathcal{O}}
\newcommand{\PP}{\mathcal{P}}
\newcommand{\QQ}{\mathbb{Q}}
\newcommand{\bQQ}{\overline{\QQ}}
\newcommand{\RR}{\mathbb{R}}

\newcommand{\SSS}{\mathcal{S}}
\newcommand{\bTT}{\mathbf{T}} 

\newcommand{\ZZ}{\mathbb{Z}}
\newcommand{\bZZ}{\overline{\ZZ}}

\renewcommand{\i}{\iota}

\newcommand{\w}{\omega}

\newcommand{\g}{\mathfrak{g}} 
\newcommand{\m}{\mathfrak{m}}
\newcommand{\p}{\mathfrak{p}} 
\newcommand{\s}{\mathfrak{s}}

\renewcommand{\P}{\mathfrak{P}}
\renewcommand{\k}{\mathfrak{k}}

\renewcommand{\d}{\mathbf{d}} 
\newcommand{\Ss}{S_\square}

\newcommand{\la}{\langle} 
\newcommand{\ra}{\rangle}

\newcommand{\incl}{\mathrm{incl}}
\newcommand{\pr}{\mathrm{pr}} 

\newcommand{\brkt}[2]{\la #1, #2 \ra}
\newcommand{\absv}[1]{\left|#1\right|}

\newcommand {\ol}[1] {\overline{#1}}
\newcommand {\ul}[1] {\underline{#1}}
\newcommand {\wt}[1] {\widetilde{#1}}

\newcommand{\at}{\makeatletter @\makeatother}

\usepackage[textwidth=14cm]{geometry}
\usepackage{fancyhdr}
\begin{document}
\pagestyle{fancy}
\fancyhf{}
\fancyhead[RO,LE]{\footnotesize\thepage}
\fancyhead[CE]{\footnotesize\leftmark}
\fancyhead[CO]{\footnotesize $p$-ADIC ZETA INTEGRALS}
\renewcommand{\headrulewidth}{0pt}

\title{Bushnell-Kutzko types for $P$-ordinary automorphic representations on unitary groups.}
\author[D. Marcil]{David Marcil}
\date{\today}
\address{David Marcil, Department of Mathematics, Columbia University, New York, NY 10027, USA}
\email{d.marcil\at columbia.edu}

\maketitle

\begin{abstract}
    This paper generalizes a theorem of Hida in \cite{Hid98} on the structure of ordinary representations on unitary groups to $P$-ordinary representations, where $P$ is a general parabolic subgroup of some general linear group. When $P$ is minimal, we recover Hida's theorem which asserts that ordinary subspaces are 1-dimensional. While analogous $P$-ordinary subspaces are infinite-dimensional in general, we use the theory of Bushnell-Kutzko types developed in \cite{BusKut98, BusKut99} to canonically associate a finite-dimensional type to the representation (under minor assumptions) that has multiplicity one in its $P$-ordinary subspace. We simultaneous develop the theory of modular forms on unitary groups with $P$-Iwahoric level structure whose nebentypus is a type (instead of a character) and construct lattices of $P$-ordinary modular forms inside $P$-ordinary automorphic representations. We also obtain direct consequences for the dual notion of $P$-anti-ordinary forms and representations.
\end{abstract}

\tableofcontents

\hrulefill

\section*{Introduction}
In the paper \cite{EHLS}, the four authors construct a $p$-adic $L$-function for ordinary families on unitary groups. This completed a project started more than a decade earlier by three of the four authors in \cite{HLS}. This required the development of several technical results on $p$-adic differential operators, accomplished in great part by the first author in \cite{Eis12}, to obtain a more general Eisenstein measure \cite{Eis15} than the one originally constructed in \cite{HLS}. Fundamental properties of their $p$-adic $L$-function for families are obtained by carefully computing local zeta integrals related to the doubling method \cite{GPSR87} as well as local coefficients of Siegel Eisenstein series \cite{Eis15}. The most technical calculations are for local factors at places above the fixed prime $p$ and a theorem of Hida in \cite{Hid98} establishing the uniqueness (up to scalar) of ordinary vectors plays a crucial role in their analysis.

In this article, we generalize this theorem of Hida to construct a canonical finite-dimensional subspace in the space of $P$-ordinary vectors for a $P$-ordinary representation $\pi$ on a unitary group $G$. Here, $P$ is a parabolic subgroup of a product of general linear groups related to $G$. When $P$ corresponds to (a product of) upper triangular Borel subgroups, the notion of $\pi$ being ``$P$-ordinary'' coincides with the usual notion of being ``ordinary''.

This accomplishes the first step in a broader project of the author to construct a $p$-adic $L$-function for a $P$-ordinary family on $G$, directly generalizing the work of \cite{EHLS}. In upcoming work, the author plans to develop the theory of $P$-ordinary families on unitary groups, inspired by the results of \cite{Pil12} on symplectic groups, and adapt the calculations of \cite{Eis15, EHLS} using the $P$-ordinary vectors constructed here instead of ordinary vectors.

\subsubsection*{Structure of this paper} 
In Section \ref{not and conv}, we first set some notation and conventions, and review the theory of Bushnell-Kutzko types relevant for us. Then, in Section \ref{P mod forms on Sh}, we introduce level subgroups of $G(\ZZ_p)$ that are ``$P$-Iwahoric'' (of some level $r$). Using the geometry of Shimura varieties associated to $G$, this allows us to construct $P$-Iwahoric covers over them. We also introduce the relevant notation to compare the theory on $G = G_1$ and on the unitary group $G_2$ associated to its opposite Hermitian vector space. 

This sets up the background to define holomorphic, $P$-ordinary as well as anti-holomorphic, $P$-anti-ordinary representations on $G_1$ as well as dual notions on $G_2$ in later sections. Simultaneously, it leads us to a natural definition of (holomorphic and anti-holomorphic) modular forms on $G$ whose level structure at $p$ is $P$-Iwahoric and whose nebentypus is a type, instead of a 1-dimensional character. We refer to the latter as a \emph{$P$-nebentypus} to emphasize the distinction.

In Section \ref{can Pord vectors}, we introduce Hecke operators at $p$ related to $P$ and define (holomorphic) $P$-ordinary representations as the ones having simultaneous eigenvectors for all these operators with $p$-adic unit eigenvalues. Equivalently, we define a $P$-ordinary projector $e_P$ from these operators and $\pi$ is $P$-ordinary if and only if its $p$-factor $\pi_p$ contains an $e_P$-fixed vector. When this is the case, $e_P$ determines a \emph{$P$-ordinary subspace} in $\pi_p$ which is typically infinite dimensional. We use the theory of Bushnell-Kutzko types to decompose this space into a direct sum of subspaces of \emph{$P$-ordinary vectors of type $\tau$}, or $(P, \tau)$-ordinary vectors. 

Our first main result (Theorem \ref{holo canonical vector}) describes natural homomorphisms between a type $\tau$ and the corresponding space of $(P, \tau)$-ordinary vectors. This result is stated for local factors of $\pi_p$ at places above $p$. Using well-known results about types and a minor hypothesis (which the author wishes to remove in the future), our second main result (Theorem \ref{Pord structure thm}) rephrases this statement for $\pi_p$ and proves that for a canonical type $\tau$ associated to $\pi_p$, which we called the \emph{BK-type of $\pi$}, the homomorphism constructed actually provides an isomorphism between $\tau$ and the corresponding $(P, \tau)$-ordinary subspace.

Given some fixed $P$-ordinary representation $\pi$ with BK-type $\tau$ (which is a smooth irreducible representation of some $p$-adic compact Lie group contained in $P$), one can tensor $\pi$ by a character $\chi$ of $P$ (that factors through the determinant map). Then, $\pi \otimes \chi$ is now $P$-ordinary with $BK$-type $\tau \otimes \chi$. This plays a more relevant role in upcoming work of the author to construct a $p$-adic family of $P$-ordinary representations containing $\pi$ of dimension equal to the rank $d$ of the Levi subgroup of $P$. Moreover, the isomorphism above allows us to vary a fixed $(P, \tau)$-ordinary vector of $\pi$ $p$-adically in this family. Again, in upcoming work, this allows the author to adapt the crucial calculations of \cite[Section 4]{EHLS} and \cite[Section 2]{Eis15} to construct $(d+1)$-variables $p$-adic $L$-functions on $G$.

In Section \ref{P-a-ord thms}, we define the analogous objects for $P$-anti-ordinary representations on $G = G_1$. Using pairs of contragredient representations, the two notions are dual to each other and we obtain consequences about space of $P$-anti-ordinary vectors from our work in the previous section. We also prove analogous statements on $G_2$. Relying on a canonical identification between $G_1$ and $G_2$, we first obtain identical results by simply replacing $P$ with its opposite parabolic $P^\opp$. However, using standard intertwining operators, we state the analogous result with $P$ instead of $P^\opp$. As a part of this broader project on $p$-adic $L$-functions, this is purely from computational purposes. Namely, in upcoming work of the author, some Rankin-Selberg zeta integrals are evaluated involving $P$-anti-ordinary vectors on both $G_1$ and $G_2$ and the analysis is simpler when both parabolic subgroups are equal instead of opposite to one another.

Finally, in Section \ref{Pord mod forms}, we use classical comparison theorems between coherent cohomology on Shimura varieties and cohomology of Lie algebras to embed integral spaces of $P$-ordinary holomorphic modular forms as lattices inside $P$-ordinary representations. The $P$-nebentypus of these forms at $p$ is directly related to the type of the corresponding $P$-ordinary vectors.

\subsubsection*{Similar results in the literature.} Many of our results are greatly inspired by analogous statements in \cite[Sections 6 and 8]{EHLS} when $P$ is minimal (i.e. a Borel subgroup $B$). The author would like to point out that many statements are quite similar, both in content and in notation. However, the reader should keep in mind that the difference of level structure at $p$, i.e. related to $P$ here instead of $B$, makes our work a genuine generalization of their careful analysis. We try to add a subscript $P$ when relevant to emphasize the distinction but this convention is not always held, especially when the notation already involes a long list of subscripts.

Furthermore, similar notions of ``$P$-ordinary'' have been considered for symplectic group to develop ``$P$-ordinary Hida theory'' (see \cite{Pil12}) and $p$-adic $L$-functions for $P$-ordinary families (\cite{LiuRos20}). However, in both cases, the analogous definitions of $P$-Iwahori subgroups are slightly less general and as a consequence, all the Bushnell-Kutzko types involved are all 1-dimensional. Another goal of this article is to develop the theory to allow types of any dimension. 

One motivation to do so, other than for the sake of generality, is that our more general notions and definitions imply that all (holomorphic cuspidal) automorphic representations of $G$ are trivially $\GL(n)$-ordinary, where $n$ is the dimension of the Hermitian vector space associated to $G$. However, if we restrict our attention and only involve 1-dimensional types, this is no longer true. In fact, in this case, a necessary condition for $\pi$ to be $\GL(n)$-ordinary would be that its local factors at places above $p$ contained an $\SL(n)$-fixed vector. Our goal is to avoid such restrictions. With this more general notion of being $P$-ordinary, the case of $P = \GL(n)$ in our broader project leads to the construction of a 2-variable $p$-adic $L$-function associated to any (holomorphic cuspidal) automorphic representation $\pi$ of $G$.

\subsubsection*{Acknowledgments} I thank Michael Harris who first suggested that I look at the work of \cite{EHLS} and adapt it to the $P$-ordinary setting. His countless insights and comments greatly helped me to obtain the results of this article, which is roughly the first third of the thesis he supervised.

\section{Notation and conventions} \label{not and conv}
Let $\bQQ \subset \CC$ be the algebraic closure of $\QQ$ in $\CC$. For any number field $F \subset \bQQ$, let $\Sigma_F$ denote its set of complex embedding $\hom(F, \CC) = \hom(F, \bQQ)$.

Throughout this article, we fix a CM field $\KK \subset \bQQ$ with ring of integers $\OO = \OO_{\KK}$. Let $\KK^+$ be the maximal real subfield of $\KK$ and denote its ring of integers as $\OO^+ = \OO_{\KK^+}$. Let $c \in \Gal(\KK/\KK^+)$ denote complex conjugation, the unique nontrivial automorphism. Given a place $v$ of $\KK$, we usually denote $c(v)$ as $\ol{v}$.

Let $\ZZ(1) \subset \CC$ be the kernel of the exponential map $\exp : \CC \to \CC^\times$, a free rank one $\ZZ$-module with noncanonical basis $2\pi\sqrt{-1}$. For any commutative ring $R$, denote $R \otimes \ZZ(1)$ by $R(1)$.


\subsection{CM types and local places} \label{CM types}
Fix an integer prime $p$ that is unramified in $\KK$. Throughout this paper, we assume the following :

\begin{hypothesis}\label{above p split}
	Each place $v^+$ of $\KK^+$ above $p$ split as $v^+ = v \bar{v}$ in $\KK$.
\end{hypothesis}

This hypothesis plays a crucial role in our analysis of the local factors at place above $p$ of the automorphic representations considered in later sections. 

Fix an algebraic closure $\bQQ_p$ of $\QQ_p$ and an embedding $\incl_p : \bQQ \hookrightarrow \bQQ_p$. Define
\[
	\bZZ_{(p)} = \{z \in \bQQ : \v_p(\incl_p(z)) \geq 0 \} \ ,
\]
where $\v_p$ is the canonical extension to $\bQQ_p$ of the normalized $p$-adic valuation on $\QQ_p$. 

Let $\CC_p$ be the completion of $\bQQ_p$. The map $\incl_p$ yields an isomorphism between its valuation ring $\OO_{\CC_p}$ and the completion of $\bZZ_{(p)}$ which extends to an isomorphism $\i : \CC \xrightarrow{\sim} \CC_p$.

Fix an embedding $\i_\infty : \QQ \hookrightarrow \CC$ such that $\incl_p = \i \circ \i_\infty$ and identify $\bQQ$ with its image in both $\CC$ and $\CC_p$.

Given $\sigma \in \Sigma_{\KK}$, the embedding $\incl_p \circ \sigma$ determines a prime ideal $\p_\sigma$ of $\Sigma_\KK$. There may be several embeddings inducing the same prime ideal. Similarly, given a place $w$ of $\KK$, let $\p_w$ denote the corresponding prime ideal of $\OO$.

Under Hypotesis \ref{above p split}, for each place of $v^+$ of $\KK^+$ above $p$, there are exactly two primes of $\OO$ above $v^+$. Fix a set $\Sigma_p$ containing exactly one of these prime ideals for each such place $v^+$. The set $\Sigma = \{\sigma \in \Sigma_\KK \mid \p_\sigma \in \Sigma_p\}$ is a CM type of $\KK$ (see \cite[p.202]{Kat78}).


\subsection{Bushnell-Kutzko Types} \label{BK types}
To discuss the local theory of $P$-ordinary representations in later sections, let us recall the theory of Bushnell-Kutzko types and covers, adapting the notions of \cite{BusKut98} and \cite[Section 3]{Lat21} to our setting.

Fix a place $w$ of $\OO$ and write $F = \KK_w$. Similarly, let $\OO_F$ denote $\OO_{\KK_w}$. Let $G = \GL_n(F)$ for some integer $n \geq 1$.

\subsubsection{Parabolic inductions and Jacquet modules.}
For any parabolic subgroup $P$ of $G$, let $L$ and $P^u$ denote its Levi factor and unipotent radical, respectively. Let $\delta_P : P \to \CC^\times$ denote its modulus character.

Recall that $\delta_P$ factors through $L$. Moreover, if $P$ is the standard parabolic subgroup associated to the partition $n = n_1 + \ldots + n_t$, one has
\begin{equation}
	\delta_P(l) = \prod_{j=1, \ldots, t} \absv{\det(l_j)}^{-\sum_{i < j} n_i + \sum_{i > j} n_i}
\end{equation}
for any $l = (l_1, \ldots, l_t)$ in $L = \prod_{j=1}^t \GL_{n_j}(F)$. In particular, $\delta_P$ agrees with $\delta_B$ on the center $Z(L)$ of $L$, where $B$ is the Borel upper triangular subgroup (associated to the partition $n = 1 + \ldots + 1$).

Given a smooth representation $(\sigma, W)$ of $L$, let $\Ind_P^{G}(\sigma, W)$ denote the classical (unnormalized) parabolic induction functor from $P$ to $G$. Moreover, given a representation $(\pi, V)$ of $G$, let $(\pi_P, V_P)$ denote the classical $P$-Jacquet functor. We often consider $\sigma$ and $\pi_P$ as both representations of $L$ and $P$ without comments.

\begin{definition}
	The \emph{normalized} parabolic induction functor is
	\[
		\ind{P}{G}(\sigma, W) = \Ind_P^{G}(\sigma \otimes \delta_P^{1/2}, W)
	\]
	and the \emph{normalized} Jacquet functor is
	\[
		\r{P}{G}(\pi, V) = (\pi_P \otimes \delta_P^{-1/2}, V_P)
	\]
\end{definition}

We often simply write $\ind{P}{G} \sigma$ (resp. $\ind{P}{G} W$) and $\r{P}{G} \pi$ (resp. $\r{P}{G} V$) when the associated vector space (resp. representation) is clear from context.

The Frobenius reciprocity theorem \cite[Theorem 2.4.1]{Cas95} states
\[
    \hom_{G}
    (
        \pi,
        \ind{P}{G} \sigma
    ) = 
    \hom_P
    (
        \r{P}{G} \pi, 
        \sigma
    )
\]

\subsubsection{Supercuspidal support} \label{sc support}
A theorem of Jacquet (see \cite[Theorem 5.1.2]{Cas95}) implies that given any irreducible representation $\pi$ of $G$, one may find a parabolic subgroup $P$ of $G$ with Levi subgroup $L$ and a supercuspidal representation $\sigma$ of $L$ such that $\pi \subset \ind{P}{G} \sigma$. .

The pair $(L, \sigma)$ is uniquely determined by $\pi$, up to $G$-conjugacy and one refers to this conjugacy class as the \emph{supercuspidal support} of $\pi$.

Consider two pairs $(L, \sigma)$ and $(L', \sigma')$ consisting of a Levi subgroup of $G$ and one of its supercuspidal representation. One says that they are \emph{$G$-inertially equivalent} if there exists some $g \in G$ such that $L' = g^{-1}Lg$ and some unramified character $\chi$ of $L'$ such that $\prescript{g}{}{\sigma} \cong \sigma' \otimes \chi$, where $\prescript{g}{}{\sigma}(x) = \sigma(gxg^{-1})$. We write $[L, \sigma]_G$ for the $G$-inertial equivalence class of $(L, \sigma)$ and let $\mathfrak{B}(G)$ for the set of such classes.

For each $\s \in \mathfrak{B}(G)$, let $\Rep^\s(G)$ denote the full subcategory of $\Rep(G)$ whose objects are the representations such that all their irreducible subquotients have inertial equivalence class $\s$.

The Bernstein-Zelevinsky geometric lemma (see \cite[Section VI.5.1]{Ren10}) implies that $\ind{P}{G} \sigma$ is an object of  $\Rep^\s(G)$, where $\s = [L, \sigma]_G$.

\begin{definition}[\cite{BusKut98}]
    Let $J$ be a compact open subgroup of $G$ and $\tau$ be an irreducible representation of $J$. Let $\Rep_\tau(G)$ denote the full subcategory of $\Rep(G)$ whose objects are the representations generated over $G$ by their $\tau$-isotypic subspace. We say that $(J, \tau)$ is an $\s$-type if $\Rep_\tau(G) = \Rep^\s(G)$. 
\end{definition}

If $\pi$ is an irreducible supercuspidal representation of $G$ with inertial support $\s$, then one can easily construct an $\s$-type $(J, \tau)$, see \cite[Section 5]{BusKut98}. By \cite[Proposition 5.6]{BusKut98}, the complex vector space $\hom_{J}(\tau, \pi)$ is 1-dimensional.

Furthermore, it follows from \cite[Theorem 1.3]{Pas05} that there exists a unique (up to isomorphism) representation $\tau$ of $K = G(\OO_F)$ such that $(K, \tau)$ is an $\s$-type. We refer to this unique ``maximal'' type of $\s$ as the \emph{BK-type} of the supercuspidal representation $\pi$.

\section{$P$-nebentypus of modular forms on unitary Shimura varieties.} \label{P mod forms on Sh}

In this section, we introduce the main algebraic groups of interest for this paper. We are mostly concerned about its structure over $\ZZ_p$ and construction of particular $p$-adic parabolic subgroups $P$. Furthermore, we analyse the geometry of the associated Shimura varieties and consider automorphic vector bundles over them (of a fixed weight $\kappa$ and \emph{$P$-nebentypus} $\tau$). This allows to discuss the theory of modular forms whose $p$-level structure is ``$P$-Iwahoric''. This sets up the background to discuss (holomorphic and anti-holomorphic) cuspidal representations that have a particular behavior under the action of the $P$-Iwahori subgroup in the next sections. We follow the standard approach and material of \cite{Hid04, CEFMV, EHLS}.

\subsection{Unitary Groups} \label{unitary groups} 

Let $V$ be a finite-dimensional $\KK$-vector space, equipped with a pairing $\brkt{\cdot}{\cdot}_{V}$ that is Hermitian with respect to the quadratic extension $\KK/\KK^+$. Write $n = \dim_\KK V$.

Let $\delta \in \OO$ be totally imaginary and prime to $p$ and define $\brkt{\cdot}{\cdot} = \tr_{\KK / \QQ}(\delta \brkt{\cdot}{\cdot}_{V})$. This choice of $\delta$ and our Hypothesis \eqref{above p split} ensure the existence of an $\OO$-lattice $L \subset V$ such that the restriction of $\brkt{\cdot}{\cdot}$ to $L$ is integral and yields a perfect pairing on $L \otimes \ZZ_p$.

For each $\sigma \in \Sigma_{\KK}$, let $V_{\sigma}$ denote $V \otimes_{\KK, \sigma} \CC$. It has a $\CC$-basis diagonalizing the pairing $\brkt{\cdot}{\cdot}$. The only eigenvalues must be $\pm 1$, say that $1$ (resp. $-1$) has multiplicity $r_{\sigma}$ (resp. $s_{\sigma}$). We order the basis so that the $+1$-eigenvectors appear first. Fixing such a basis, let $h_{\sigma} : \CC \to \End_{\RR}(V_{\sigma})$ be $h_{\sigma} = \diag(z 1_{r_{\sigma}}, \bar{z} 1_{s_{\sigma}})$.

Let $h = \prod_{\sigma \in \Sigma} h_{\sigma} : \CC \to \prod_{\sigma \in \Sigma} \End_{\RR}(V_{\sigma})$ and assume that $h$ is \emph{standard} (see \cite[Section 2.3.2]{EHLS}). Since $\Sigma$ is a CM type of $\KK$, one has a canonical identification
\[
    \prod_{\sigma \in \Sigma}
        \End_{\RR}(V_{\sigma}) 
    = 
        \End_{\KK^+ \otimes \RR}(V \otimes \RR)
\]

The tuple 
$
    \PP = 
        (
            \KK, c, \OO, L, 
            2\pi\sqrt{-1}\brkt{\cdot}{\cdot}, h
        )
$ 
is a PEL datum of unitary type, as defined in \cite[Section 2.1-2.2]{EHLS}. It has an associate group scheme $G = G_\PP$ over $\ZZ$ whose $R$-points are
\[
	G(R) = \{ (g, \nu) \in \GL_{\OO \otimes R}(L \otimes R) \times R^\times \mid \brkt{gx}{gy} = \nu \brkt{x}{y}, \forall x, y \in L \otimes R \},
\]
for any commutative ring $R$. In particular, $G_{/\QQ}$ is a reductive group. Moreover, the assumptions on $p$ imply that $G_{/\ZZ_p}$ is smooth and $G(\ZZ_p)$ is a hyperspecial maximal compact of $G(\QQ_p)$.

\subsubsection{Hodge structure.} \label{Hodge structure}
The homomorphism $h$ determines a pure Hodge structure of weight $-1$ on $V_\CC = L \otimes \CC$, i.e. $V = V^{-1, 0} \oplus V^{0,-1}$ and $h(z)$ acts as $z$ on $V^{-1,0}$ and as $\bar{z}$ on $V^{0,-1}$. In particular, the $\OO \otimes \CC$-submodule $V^0 \subset V$ defined as the degree 0 piece of the corresponding Hodge filtration is simply $V^{-1,0}$. 

For each $\sigma \in \Sigma_{\KK}$, let $a_{\sigma} = \dim_{\CC} (V^0 \otimes_{\OO \otimes \CC, \sigma} \CC)$ and $b_{\sigma} = n - a_{\sigma}$. The \emph{signature} of $h$ is defined as the collection of pairs $\{ (a_{\sigma}, b_{\sigma} )_{\sigma \in \Sigma_\KK} \}$. Throughout this paper, we assume :

\begin{hypothesis}[Ordinary hypothesis]
	For all embeddings $\sigma, \sigma' \in \Sigma_{\KK}$, if $\p_{\sigma} = \p_{\sigma'}$, then $a_{\sigma} = a_{\sigma'}$.
\end{hypothesis}

Therefore, given a place $w$ of $\KK$ above $p$, one can define $(a_{w}, b_{w}) := (a_{\sigma}, b_{\sigma})$, where $\sigma \in \Sigma_\KK$ is any embedding such that $\p_\sigma = \p_w$. Observe that $(a_{\sigma}, b_{\sigma}) = (r_{\sigma}, s_{\sigma})$ is $\sigma \in \Sigma$. Otherwise, one has $(a_{\sigma}, b_{\sigma}) = (s_{\sigma}, r_{\sigma})$.

\subsection{Structure of $G$ over $\ZZ_p$.}
In this section, we introduce the preliminary notions that allows us to later study automorphic representations that are \emph{ordinary with respect to some parabolic subgroup} of $G$.

\subsubsection{Comparison to general linear groups.}
Consider the factorization $\OO \otimes \ZZ_p = \prod_{w \mid p} \OO_{w}$ as the product runs over all primes $w$ of $\KK$ above $p$. This induces a decomposition $L \otimes \ZZ_p = \prod_{w \mid p} L_w$ and a canonical $\ZZ_p$-isomorphism
\begin{equation}
    \GL_{
        \OO \otimes \ZZ_p
    }(L \otimes \ZZ_p) 
        \xrightarrow{\sim} 
    \prod_{w \mid p} 
        \GL_{\OO_w}(L_w), 
            \ \ \ \ g \mapsto (g_{w}) \ .
\end{equation}

One obtains an isomorphism
\begin{equation} \label{prod G over Zp}
    G_{/\ZZ_p} 
        \xrightarrow{\sim} 
    \Gm \times 
    \prod_{w \in \Sigma_p} 
        \GL_{\OO_w}(L_w), 
            \ \ \ \ (g, \nu) \mapsto (\nu, (g_{w})) \ .
\end{equation}

Our assumption above about the pairing $\brkt{\cdot}{\cdot}$ implies that for each $w \mid p$, there is an $\OO_{B,w}$-decomposition of $L_w = L_w^+ \oplus L_w^-$ such that 
\begin{enumerate}
	\item $\rk_{\OO_w}{L_w^+} = a_{w}$ and $\rk_{\OO_w}{L_w^-} = b_{w}$;
	\item Upon restricting $\brkt{\cdot}{\cdot}$ to $L_w \times L_{\bar{w}}$, the annihilator of $L_w^\pm$ is $L_{\bar{w}}^{\pm}$. Hence, one has a perfect pairing $L_w^+ \oplus L_{\bar{w}}^- \to \ZZ_p(1)$, again denoted $\brkt{\cdot}{\cdot}$.
\end{enumerate}

Fix dual $\OO_w$-bases (with respect to the perfect pairing above) for $L_w^+$ and $L_{\bar{w}}^-$. They yield identifications
\begin{equation} \label{GL(ei Lw) basis}
    \begin{tikzcd}
            \GL_{\OO_w}(L_w^+) 
                \arrow[r, "\cong"] 
        &
            \GL_{a_{w}}(\OO_w) 
        &
            \GL_{b_{\bar{w}}}(\OO_{\bar{w}}) 
                \arrow[r, "\cong"] 
        &
            \GL_{\OO_w}(L_{\bar{w}}^-)
    \end{tikzcd}
\end{equation}
as well as an isomorphism $\GL_{\OO_w}(L_w) \cong \GL_{n}(\OO_w)$ such that the obvious map
\[
	\GL_{\OO_w}(L_w^+) \times \GL_{\OO_w}(L_w^-) \hookrightarrow \GL_{\OO_w}(L_w)
\]
is simply the diagonal embedding of block matrices.

Let $H := \GL_{\OO \otimes \ZZ_p}(L^+)$. Then, the identification \eqref{GL(ei Lw) basis} above induces a canonical isomorphism
\begin{equation} \label{def H}
    H \cong
    \prod_{w \mid p}
        \GL_{a_{w}}(\OO_w) = 
    \prod_{w \in \Sigma_p}
        \GL_{a_{w}}(\OO_w) \times 
        \GL_{b_{w}}(\OO_w)
\end{equation}

\subsubsection{Parabolic subgroups of $G$ over $\ZZ_p$.} \label{level at p}
For $w \mid p$, let 
\[
	\d_{w} = \left( n_{w,1}, \ldots, n_{w, t_{w}} \right)
\]
be a partition of $a_{w} = b_{\bar{w}}$. Let $P_{\d_{w}} \subset \GL_{a_{w}}(\OO_w)$ denote the standard parabolic subgroup corresponding to $\d_{w}$. 

Let $P_H \subset H$ be the $\ZZ_p$-parabolic that corresponds to the products of all the $P_{\d_{w}}$ via the isomorphism \eqref{def H}. We denote the unipotent radical of $P_H$ by $P_H^u$. 

We work with the Levi factor $L_H = P_H / P_H^u$ of $P_H$ as well as its maximal subtorus $T_H$. Note that $T_H$ does not depend on the choice of partitions. Furthermore, elements of $L_H$ are identified with collections of block-diagonal matrices, with respect to the partitions $\d_{w}$, via \eqref{def H}. 

Let $P^+ \subset G_{/\ZZ_p}$ be the parabolic subgroup that stabilizes $L^+$ and such that
\begin{equation} \label{def P+}
    P^+ 
        \twoheadrightarrow 
    \Gm \times P_H \subset \Gm \times H
\end{equation}
where the map to the first factor is the similitude character $\nu$ and the map to the second factor is projection to $H$. 

For $w \in \Sigma_p$, let $P_{w}$ be the parabolic subgroup of $\GL_{n}(\OO_w)$ given by
\begin{equation} \label{local parabolics}
    P_{w} = 
    \left\{
        \begin{pmatrix}
		A & B \\
		0 & D
        \end{pmatrix} 
            \in \GL_{n}(\OO_w) \mid
                A \in P_{\d_w}, 
                D \in P^{\opp}_{\d_{\ol{w}}}
    \right\} 
\end{equation}
and set $P = \prod_{w \in \Sigma_p} P_{w}$.

We naturally identify $P$ as a subgroup of $G_{/\ZZ_p}$. Let $P^u$ be the unipotent radical of $P$, $L_P = P/P^u$ be its Levi factor and $T_P$ be its maximal subtorus. The projection $P^+ \twoheadrightarrow \Gm \times P_H$ induces a natural isomorphism $L_P \cong L_H$. Its restrictions to maximal subtori yields the identity map $T_P = T_H$.

\begin{remark} \label{trivial partition}
    The trivial partition of $a_{w}$ is $(1, \ldots, 1)$ (of length $t_{w} = a_{w}$). If the partitions $\d_{w}$ and $\d_{\bar{w}}$ are both trivial, we write $B_{w}$ instead of $P_{w}$. In that case, $L_B = T_B = T_H$.
\end{remark}

Our choices of bases above imply that under the isomorphisms \eqref{prod G over Zp} and \eqref{GL(ei Lw) basis}, $P^+$ corresponds to
\begin{equation}
	P^+ \xlongrightarrow{\sim} \Gm \times P \ .
\end{equation}

\begin{definition} \label{def PIwahori}
    We define the $P$-Iwahori subgroup of $G$ of level $r \geq 0$ as
    \[
        I_r^0 = I_{P,r}^0 :=
        \left\{
            g \in G(\ZZ_p) \mid 
                g \text{ mod } p^r 
                \in P^+(\ZZ_p/p^r \ZZ_p)
        \right\}
    \]
    and the pro-$p$ $P$-Iwahori subgroup $I_r = I_{P,r}$ of $G$ of level $r$ as 
    \[
        I_r = I_{P,r} :=
        \left\{
            g \in G(\ZZ_p) \mid 
                g \text{ mod } p^r 
                \in (\ZZ_p/p^r\ZZ_p)^\times 
                \times P^u(\ZZ_p/p^r \ZZ_p)
        \right\}.
    \]

    Note that for $r = 0$, we simply have $I_{P,0} = I_{P,0}^0 = G(\ZZ_p)$.
\end{definition}

\begin{remark}
    We refrain from referring to $I_r^0$ as a \emph{parahoric} subgroup of $G$. This terminology is usually reserved for stabilizers of points in Bruhat-Tits building. We make no attempt here to introduce our construction from the point of view of these combinatorial and geometric structures.
\end{remark}

The inclusion of $L_P(\ZZ_p)$ in $I_r^0$ yields a canonical isomorphism 
\begin{equation}
    L_P(\ZZ_p/p^r\ZZ_p) \xrightarrow{\sim} I_r^0/I_r \ .
\end{equation}

For each $w \in \Sigma_p$, one similarly defines $I_{w, r}^0$ and $I_{w, r}$ by replacing $P^+$ by $P_{w}$ and working in $\GL_{n}(\OO_w)$ instead of $G(\ZZ_p)$. Let
\[
    I_r^{\GL} = 
    \prod_{w \in \Sigma_p}
        I_{w, r} 
        \ \ \ \text{and} \ \ \ 
    I_r^{0,\GL} = 
    \prod_{w \in \Sigma_p} 
        I_{w, r}^0 \ ,
\]
so that $I_{r}$ and $I_{r}^0$ correspond to $\ZZ_p^\times \times I_{P,r}^{\GL}$ and $\ZZ_p^\times \times I_{P,r}^{0, \GL}$ respectively, via the isomorphisms \eqref{prod G over Zp} and \eqref{GL(ei Lw) basis}.

\begin{remark}
    Later, we will consider various modules with an action of $G$ and define ``$P$-ordinary'' submodules. Technically, it would be more accurate to refer to them as $P^+$-ordinary submodules. Similarly, the groups defined above could be called (pro-$p$) $P^+$-Iwahori subgroups. In any case, there should not be any confusion between $P$ and $P^+$.
\end{remark}

\subsubsection{Conventions for the opposite unitary group of $G$.} \label{P1 P2 P3 P4}
Consider the PEL datum $\PP = (\KK, c, \OO, L, \brkt{\cdot}{\cdot}, h)$ of unitary type associated to a finite-dimensional hermitian $\KK$-vector space $(V, \brkt{\cdot}{\cdot})$ as above. Recall that there is a fixed $\OO \otimes \ZZ_p$-decomposition $L \otimes \ZZ_p = L^+ \oplus L^-$.

We sometimes write $\PP_1$ for $\PP$ and similarly set $L_1 := $, $\brkt{\cdot}{\cdot}_1 := \brkt{\cdot}{\cdot}$ and $h_1 := h$. Define
\[
    \PP_2 = 
    (
        \KK, c, \OO, L_2, 
        \brkt{\cdot}{\cdot}_2, 
        h_2
    ) := 
    (
        \KK, c, \OO, L, 
        -\brkt{\cdot}{\cdot}, 
        h \circ \overline{(\cdot)}
    )
\]
which is clearly the datum associated to $V$ but equipped with the opposite Hermitian pairing $-\brkt{\cdot}{\cdot}$. When we wish to distinguish those PEL datum, we write $G_1 := G_{\PP_1}$ and $G_2 := G_{\PP_2}$. One has an obvious canonical identification $G_1(\AA) = G_2(\AA)$.

All of the definitions above can therefore be made with $\PP_2$ instead of $\PP_1$. To compare the relevant results on these two groups, we choose the fixed $\OO \otimes \ZZ_p$-decomposition for $L_2 \otimes \ZZ_p = L_2^+ \oplus L_2^-$ to be $L_2^\pm := L^\mp$. Furthermore, the signature of $G_2$ at $w \in \Sigma_p$ is now $(a_{\ol{w}}, b_{\ol{w}}) = (b_w, a_w)$. Therefore, when working with $G_2$, we fix the partition of $a_{\ol{w}}$ to be the partition $d_{\ol{w}}$ of $b_w$ chosen above.

\begin{remark}
    We often refer to $(V, \brkt{\cdot}{\cdot})$ simply by $V$ and $(V, -\brkt{\cdot}{\cdot})$ simply by $-V$. The objects associated to each of them sometimes have subscripts $V$ or $-V$ to emphasize the relevant PEL datum. This convention will be reminded several times throughout the article to avoid confusion, especially in Section \ref{Pord and Paord on G_2}.
\end{remark}

\subsection{Unitary Shimura varieties of level $I_{P,r}$ at $p$.} \label{mod space and Shi var}
The results of Sections \ref{can Pord vectors} and \ref{P-a-ord thms} can be obtained while only working with moduli spaces associated to $\PP$ over $F$, the reflex field of $\PP = \PP_1$. However, in Section \ref{Pord mod forms}, we use these results to compare ``$P$-ordinary subspaces'' (to be defined later) with $p$-integral spaces of modular forms. Therefore, in this section, we introduce the relevant spaces over $F$ and over $\OO_F \otimes \ZZ_{(p)}$ simultaneously, where $\OO_F$ denotes the ring of integers of $F$. 

\begin{remark}
    In the $p$-integral case, we assume first that our level $K$ is hyperspecial at $p$, so our treatment here follows \cite[Section 2.2]{EHLS} and introduces the notions relevant to our situation. However, in Section \ref{Shi var level Kr}, we introduce level structures at $p$ that are more general than the one considered in \cite{EHLS}.
\end{remark}

Let $\square = \{p\}$ or $\emptyset$ and define $\Ss = \OO_F \otimes \ZZ_{(\square)}$. Let $K^\square \subset G(\AA_f^\square)$ be any open compact subgroup and set
\[
    K = \begin{cases}
        K^\square, & \text{if } \square = \{0\}, \\
        G(\ZZ_p)K^\square, & \text{ otherwise.}
    \end{cases}
\]

Then, one may define the moduli problem $\MM_{K, \square} = \MM_{K, \square}(\PP)$ as the functor that assigns to any locally noetherian $S_\square$ scheme $T$ the set of equivalence classes of quadruples $\ul{A} = (A, \lambda, \iota, \alpha)$, where
\begin{enumerate}
    \item $A$ is an abelian scheme over $A$;
    \item $\lambda : A \to A^\vee$ is a polarization. If $\square = \{p\}$, this polarization is prime-to-$p$;
    \item $\iota : \Ss \hookrightarrow \End_T A \otimes \ZZ_{(\square)}$ such that $\iota(b)^\vee \circ \lambda = \lambda^\vee \circ \iota(\ol{b})$;
    \item $\alpha$ is a $K^\square$-level structure, see \cite[Section 2.1]{EHLS};
    \item $\Lie_T A$ satisfies the Kottwitz determinant condition defined by $(L \otimes R, \brkt{\cdot}{\cdot}, h)$, see \cite[Definition 1.3.4.1]{Lan13};
\end{enumerate}
and two quadruples $(A, \lambda, \iota, \alpha)$ and $(A', \lambda', \iota', \alpha')$ are equivalent if there exists some prime-to-$\square$ isogeny $f : A \to A'$ such that
\begin{enumerate}
    \item $\lambda$ and $f^\vee \circ \lambda' \circ f$ are equal, up to multiplication by some positive element in $\ZZ_{(\square)}^\times$;
    \item $\iota'(b) \circ f = f \circ \iota(b)$, for all $b \in \OO_F$;
    \item $\alpha' = f \circ \alpha$.
\end{enumerate}

If $K$ is \emph{neat} (see \cite[Definition 1.4.1.8.]{Lan13}), then there exists a smooth, quasi-projective $\Ss$-scheme that represents this moduli problem, which we still denote by $\MM_{K, \square}$. One readily sees that $\MM_{K, \emptyset}$ is canonically isomorphic to the base change of $\MM_{K, \{p\}}$ from $\OO_F \otimes \ZZ_{(p)}$ to $F$. Therefore, when the base ring $\Ss$ is clear from context, we simply write $\MM_K$ for $\MM_{K, \square}$.

\subsubsection{Toroidal compactifications.}
We recall the existence of toroidal compactifications of the moduli spaces above constructed in \cite{Lan13}. When $\square = \{p\}$, these generalizes the known toroidal compactifications for $\square = \emptyset$. Note that these are associated to \emph{smooth projective polyhedral cone decompostions}. Since the exact definition of the later plays no role in this article, we do not introduce this notion precisely. 

The only properties relevant for us are that given such a polyhedral cone decomposition $\Omega$, there exists a smooth toroidal compactification $\MM^\tor_{K, \Omega}$ of $\MM_K$ over $S_\square$, for both $\square = \emptyset$ and $\{p\}$, and that there exists a partial ordering on the set of such $\Omega$'s by \emph{refinements}. Given two polyhedral cone decompositions $\Omega$ and $\Omega'$, if $\Omega'$ refines $\Omega$, then there is a canonical proper surjective map $\pi_{\Omega', \Omega} : \MM_{K, \Omega'}^\tor \to \MM_{K, \Omega}^\tor$ which restricts to the identity on $\MM_K$. We denote the tower $\{\MM_{K, \Omega}^\tor\}$ by $\MM_K^\tor$. We often refer to the tower as if it were a single scheme and do not emphasize the specific compatible choices of $\Omega$ in some constructions. See \cite[Section 2.4]{EHLS} for more details.

\subsubsection{Compactified Shimura varieties of level $K_{P,r}$.} \label{Shi var level Kr}
Over the reflex field $F$, the moduli space $\MM_{K}(\PP)$ is the union of finitely many copies of the canonical model of the Shimura variety associated to $(G, X_\PP)$, where $X_\PP$ denote the $G(\RR)$-conjugacy class of $h$, see \cite[Section 8]{Kot92} for details.

More precisely, let $V^{(1)}, \ldots, V^{(k)}$ be representatives for the isomorphism classes of all hermitian vector spaces that are locally isomorphic to $V$ at every place of $\QQ$. As explained in \cite[Section 2.3.2]{CEFMV}, it is well-known that there are finitely many such classes, in fact $k = |\ker^1(\QQ, G)|$, where 
\[
    \ker^1(\QQ, G) = \ker\left( H^1(\QQ, G) \to \prod_v H^1(\QQ_v, G) \right).
\]

Then, $\MM_{K} = \MM_{K, \emptyset}$ is the disjoint union of isomorphic $F$-schemes $\MM_{K, V^{(j)}}$ naturally indexed by the $V^{(j)}$. Assume that $V^{(1)} = V$ and denote the scheme-theoretic of $\MM_{K, V}$ in $\MM_{K, \square}$ by $\level{K}{\Sh_\square}(V)$. Again, we often simplify the notation to $\level{K}{\Sh}(V)$ (or even $\level{K}{\Sh}$) when the choice of $\square = \emptyset$ or $\{p\}$ is clear from context. In particular, $\level{K}{\Sh}$ is a smooth, quasi-projective $S_\square$-scheme. We refer to $\level{K}{\Sh}$ as a \emph{Shimura variety} of level $K$ (associated to $\PP$) and $\MM_K$ as a \emph{moduli space}.

In what follows, we work with $\square = \{p\}$, hence $K = G(\ZZ_p)K^p$ as in the beginning of Section \ref{mod space and Shi var}. We now introduce a more general level structure at $p$. To do so, we first need to introduce covers of $\MM_K$ and $\MM_K^{\tor}$.

Let $\ul{\Ab} = (\Ab, \lambda, \iota, \alpha)$ be the universal abelian scheme over $\MM_K$. Using \cite[Theorem 6.4.1.1]{Lan13}, $\Ab$ can be extended to a semiabelian scheme over $\MM_K^\tor$ that is part of a degenerating family and which we still denote $\Ab$. By \cite[Theorem 3.4.3.2]{Lan13}, there exists a dual semiabelian scheme $\Ab^\vee$ together with homomorphisms $\Ab \to \Ab^\vee$, $\OO_F \otimes \ZZ_{(p)} \to \End_{\MM_K^{\tor}} \Ab$ and a $K^{(p)}$-level structure on $\Ab$ that extend $\lambda$, $\iota$ and $\alpha$ respectively.

Define an $\OO_F \otimes \ZZ_{(p)}$-scheme $\overline{\MM}_{K_r}$ over $\MM_K^{\tor}$ whose $S$-points classify the $P^u_H(\ZZ_p)$-orbits of $\OO \otimes \ZZ_p$-injections $\phi : L^+ \otimes \mu_{p^r} \hookrightarrow \Ab^\vee[p^r]_{/S}$ of group schemes with image an isotropic subgroup scheme. Let $\MM_{K_r}$ denote its pullback over $\MM_K$. We have the commutative diagram 
\[
\begin{tikzcd}
    \MM_{K_r}  
        \arrow[r, hook] \arrow[d] &
    \ol{\MM}_{K_r} 
        \arrow[d] \\
    \MM_K
        \arrow[r, hook] &
    \MM^{\tor}_K
\end{tikzcd}
\]
where the vertical arrows are $\LL_r$-torsors, where $\LL_r$ denotes $L_P(\ZZ_p/p^r\ZZ_p) = L_H(\ZZ_p/p^r\ZZ_p)$.

After base change from $\OO_F \otimes \ZZ_{(p)}$ to $F$, a choice of basis of $\ZZ_p(1)$ induces a canonical identification between $\MM_{K_{r /F}}$ and the moduli space $(\MM_{I_rK^p})_{/F}$. Moreover, the normalization of $(\MM^\tor_{K})_{/F}$ in $(\MM_{K_{r}})_{/F}$ is $\overline{\MM}_{K_{r /F}}$. In other words, given any open compact subgroup $K^p \subset G(\AA_f^p)$, we may define $K_r = I_rK^p$ and there should be no confusion when working over $S_{\square}$, for $\square = \{p\}$ or $\emptyset$. We sometimes write $K_{P,r}$ instead of $K_r$ if we want to emphasize its dependence on $P$.

To define modular forms of level $K_r$, we only need to work with the components over $\level{K}{\Sh}$. More precisely, for any polyhedral cone decomposition $\Omega$, denote the scheme-theoretic closure of $\level{K}{\Sh}$ in $\MM_{K,\Omega}^{\tor}$ by $\level{K}{\Sh}^{\tor}_{\Omega}$. Again, we denote the tower $\{\level{K}{\Sh}^{\tor}_\Omega\}_\Omega$ by $\level{K}{\Sh}^{\tor}$ and describe our construction as if this tower was a single scheme. In particular, we have a canonical inclusion (of towers) $s_K : \level{K}{\Sh}^{\tor} \hookrightarrow \MM_K^{\tor}$ in the obvious sense. Its restriction to $\level{K}{\Sh}$ is the natural inclusion $\level{K}{\Sh} \hookrightarrow \MM_K$ described above, which we denote by $s_K$ again.

As discussed in \cite[Sections 3-4]{Lan12} and \cite[Section 2.4]{EHLS}, this is a smooth toroidal compactification of $\level{K}{\Sh}$. Furthermore, over $F$ (i.e. when $\square = \emptyset$), it is equal to the usual toroidal compactification of the canonical model of the Shimura variety associated to $(G, X_\PP)$.

Define $\level{K_r}{\Sh}$ (resp. $\level{K_r}{\ol{\Sh}}$) as the pullback of $\MM_{K_r}$ (resp. $\ol{\MM}_{K_r}$) via $s_K$, i.e. we have the commutative diagrams
\[
\begin{tikzcd}
    \level{K_r}{\Sh}
        \arrow[r, hook] \arrow[d] &
    \MM_{K_r} 
        \arrow[d] & &
    \level{K_r}{\ol{\Sh}} 
        \arrow[r, hook] \arrow[d] &
    \ol{\MM}_{K_r} 
        \arrow[d] \\
    \level{K}{\Sh}
        \arrow[r, hook] &
    \MM_K & &
    \level{K}{\Sh^\tor} 
        \arrow[r, hook] &
    \MM^\tor_K
\end{tikzcd}
\]

By abusing notation, we denote all four of the horizontal inclusions by $s_K$. All four vertical arrows are covers by $\LL_r$-torsors.


\subsubsection{Complex uniformization.}

We first recall the description of natural complex structure on $X = X_\PP$. Let $V_\CC = L \otimes \CC$ with its pure Hodge decomposition $V_\CC = V^{-1, 0} \oplus V^{0, -1}$ of weight $-1$, as in section \ref{Hodge structure}. Let $W = V/V^{0,-1}$, a space defined over the reflex field $F$ of $\PP$. 

Fix an $\Ss$-submodule $\Lambda_0$ of $W$ such that $\Lambda_0 \otimes_{\Ss} \CC = W$ and consider the $\Ss$-module $\Lambda_0^\vee = \hom_{\ZZ_{(p)}}(\Lambda_0, \ZZ_{(p)}(1))$. Define $\Lambda = \Lambda_0 \oplus \Lambda_0^\vee$ and
\begin{align*}
    \brkt{\cdot}{\cdot}_{can} : \Lambda \times \Lambda &\to \ZZ_{(p)}(1) \\
    \brkt{(f_1,x_1)}{(f_2,x_2)}_{can} &= f_2(x_1) - f_1(x_2)
\end{align*}
so that both $\Lambda_0$ and $\Lambda_0^\vee$ are isotropic submodules of $\Lambda$. One has $\brkt{bx}{y}_{can} = \brkt{x}{\ol{b}y}_{can}$, for $b \in \OO_F$.

The pair $(\Lambda, \brkt{\cdot}{\cdot}_{can})$ induces an $\Ss$-group scheme $G_0$ whose $R$-points are given by
\[
    G_0(R) = 
    \left\{ 
        (g, \nu) \in \GL_{R}(\Lambda \otimes_{\Ss} R) \times R^\times \mid \brkt{gx}{gy}_{can} = \nu \brkt{x}{y}_{can}, x, y \in \Lambda \otimes R
    \right\} \ ,
\]
for any $\Ss$-algebra $R$.

One readily checks that there is an isomorphism $V \cong \Lambda \otimes_{\Ss} \CC$ of $\CC$-vector spaces that identifies $V^{-1,0}$ (resp. $V^{0, -1}$) with $\Lambda_0 \otimes_{\Ss} \CC$ (resp. $\Lambda^\vee_0 \otimes_{\Ss} \CC$) and the pairing $\brkt{\cdot}{\cdot}$ with $\brkt{\cdot}{\cdot}_{can}$. In other words, it yields an identification between $G_{/\CC}$ and $G_{0/\CC}$.

Let $H_0 \subset G_0$ be the stabilizer of the polarization $\Lambda = \Lambda_0 \oplus \Lambda_0^\vee$. The algebraic representations of $H_0$ will describe the cohomological weights of the automorphic representations considered below. The natural projection
\[
    H_0 \to \Gm \times \GL_{\OO_F \otimes \Ss} (\Lambda_0)
\]
is an isomorphism.

Under the identification above, $H_0(\CC)$ corresponds to $C(\CC)$, where $C$ is the real algebraic subgroup of $G_{/\RR}$ whose real points $U_\infty = C(\RR)$ is the stabilizer of $h \in X$ under the conjugation action of $G(\RR)$. 

Let $P_0 \subset G_0$ be the parabolic subgroup defined as the stabilizer of $\Lambda_0$; its Levi factor is $H_0$. Then, the identification above embeds $G(\RR)/U_\infty \xrightarrow{\sim} X$ as an open subspace of $G_0(\CC)/P_0(\CC)$, which yields a complex structure on $X$. As discussed in \cite[Section 8]{Kot92}, the complex analytic space $\level{K}{\Sh}(\CC)$ is naturally isomorphic to
\[
    G(\QQ) \backslash X \times G(\AA_f) / K \ .
\]

Note that $P_{0 /\CC}$ corresponds to $P_{h /\CC}$, where $P_h \subset G_{/\RR}$ be  is the stabilizer of the Hodge filtration on $V = L \otimes \RR$ determined by $h$, as explained in Section \ref{Hodge structure}.

\subsection{Weight and $p$-type of automorphic vector bundles} \label{aut v bundle of w and t}

\subsubsection{The canonical bundles} \label{can bundles}
In this section, $\square$ can be either $\emptyset$ or $\{p\}$. In both cases, let $K = G(\ZZ_p)K^p$ and for any $r \geq 1$, let $K_r = I_rK^p$. When $\square = \emptyset$, some of the definitions below can be adapted for any level structure at $p$ but these will not be pertinent for our work.

Let $\w$ be the $\OO_{\MM_K^{\tor}}$-dual of $\Lie_{\MM_K^{\tor}} \Ab^\vee$ over $\Ss$. The Kottwitz determinant condition mentioned in the definition of the moduli problem $\MM_K(\PP)$ implies that $\w$ is locally isomorphic to $\Lambda_0^\vee \otimes_{\Ss} \OO_{\MM_K^\tor}$ over $\OO_{\KK} \otimes \OO_{\MM_K^\tor}$. Define
\[
    \EE = \isom_{\OO_{\KK} \otimes \OO_{\MM_K^\tor}}
    ( 
        (
            \w,
            \OO_{\MM_K^\tor}(1)
        ),
        (
            \Lambda_0^\vee \otimes_{S_\square} \OO_{\MM_K^\tor},
            \OO_{\MM_K^\tor}(1)
        )
    ) \ ,
\]
over $\MM_K^\tor$. The natural structure map is an $H_0$-torsor $\pi : \EE \to \MM_K^\tor$. Set $\EE_r = \EE \times_{\MM_K^\tor} \bar{\MM}_{K_r}$, an $\LL_r$-torsor of $\EE$, so
\[
\begin{tikzcd} 
    \EE_r 
        \arrow[r, "H_0"] \arrow[d, "\LL_r"] &
    \ol{\MM}_{K_r}
        \arrow[d, "\LL_r"] \\
    \EE
        \arrow[r, "H_0"] &
    \MM_K^\tor
\end{tikzcd}
\]
and denote the structure map $\EE_r \to \ol{\MM}_{K_r}$ by $\pi_r$. 

Let $\tau$ be a smooth finite-dimensional representation of $L_P(\ZZ_p)$ that factors through $\LL_r$. Let $M_\tau$ denote the associated complex vector space. In fact, there exists a finite ring extension $\Ss[\tau]$ of $\Ss$ on which $\tau$ is well defined. 

Define $\EE_{r, \tau}$ as the $\Ss[\tau]$-scheme over $\EE_r$ whose $R$-points are given by
\[
    \EE_{r, \tau}(R) = \EE_r(R) \times^{\tau} (M_\tau)_{/R} := (\EE_r(R) \times (M_\tau)_{/R})/{\sim^\tau}
\]
for any $\Ss[\tau]$-algebra $R$. The equivalence relation $\sim^\tau$ is given by
\[
    (\epsilon, m) \sim^\tau (g\epsilon, \tau(g)m) \ ,
\]
for all $\varepsilon \in \EE_r$, $m \in (M_\tau)_{/R}$ and $g \in L_H(\ZZ_p)$. Let $\pi_{r, \tau}$ be the structure map $\EE_{r, \tau} \to \ol{\MM}_{K_r}$.

\subsubsection{Weights of modular forms.} \label{rep of H0}
Let $\KK'$ be the Galois closure of $\KK$ and $\p' \subset \OO_{\KK'}$ be the prime above $p$ determined by $\i_p$. Moreover, let
\[
    \Ss^0 = \Ss \otimes_{\OO_{F, (p)}} \OO_{\KK', (\p')} = 
    \begin{cases}
        \KK' \ , & \text{if } \square = \emptyset \\
        \OO_{\KK', (\p')} \ , & \text{if } \square = \{p\} \\
    \end{cases}
\]

Over $\Ss^0$, we have an isomorphism
\begin{equation}\label{H0 over S0}
    H_{0 /\Ss^0} \xrightarrow{\sim} 
    \Gm \times 
    \prod_{\sigma \in \Sigma_\KK}
        \GL_{\OO \otimes_{\OO, \sigma} \Ss^0}(\Lambda_{0,\sigma}^\vee)
    \cong
    \Gm \times 
    \prod_{\sigma \in \Sigma_\KK}
        \GL_{b_{\sigma}}(\Ss^0) \ .
\end{equation}

Let $B_{H_0} \subset H_0$ be the Borel subgroup (defined over $\Ss^0$) that corresponds to the product of the lower-triangular Borel subgroups via the isomorphism \eqref{H0 over S0}. Let $T_{H_0} \subset B_{H_0}$ denote its maximal subtorus and let $B^u_{H_0}$ denote its unipotent radical subgroup.

Given an $\Ss^0$-algebra $R$, a character $\kappa$ of $T_{H_0}$ over $R$ is identified via the isomorphism \eqref{H0 over S0} with a tuple
\[
    \kappa = 
        (
            \kappa_0, 
            (
                \kappa_{\sigma}
            )_{
                \sigma \in \Sigma_\KK, 
            }
        ) \ ,
\]
where $\kappa_0 \in \ZZ$ and $\kappa_{\sigma} = (\kappa_{\sigma, j}) \in \ZZ^{b_{\sigma}}$. Namely, for 
\[
    t = 
        (
            t_0, 
            ( 
                \diag(
                    t_{\sigma, i, 1}, 
                    \ldots, 
                    t_{\sigma, i, b_{\sigma, i}}
                )
            )_{
                \sigma \in \Sigma_\KK
            }
        ) \in T_{H_0} \ ,
\]
one has
\[
    \kappa(t) = 
    t_0^{\kappa_0} 
    \prod_{\sigma \in \Sigma_\KK} 
    \prod_{j=1}^{b_{\sigma}}
        t_{\sigma, j}^{\kappa_{\sigma, j}}
\]

We say that $\kappa$ is \emph{dominant} if it is dominant with respect to the opposite Borel $B_{H_0}^\opp$ (of upper-triangular matrices). This is equivalent to $\kappa_{\sigma, j-1} \geq \kappa_{\sigma, j}$ for all $\sigma \in \Sigma_\KK$, $2 \leq j \leq b_{\sigma}$.

Given a dominant character $\kappa$ of $T_{H_0}$ over an $\Ss^0$-algebra $R$, extend it trivially to $B_{H_0}$. Define
\[
    W_\kappa = W_\kappa(R) = \Ind_{B_{H_0}}^{H_0} \kappa 
    = 
        \{
            \phi : H_{0_{/R}} \to \Ga \mid 
            \phi(bh) = \kappa(b)\phi(h), 
            \forall b \in B_{H_0} 
        \} \ .
\]
with its natural structure as a left $H_0$-module via multiplication on the right. Since $H_0$ is the Levi factor of $P_0$, we inflate it to an irreducible algebraic representation of $P_0$.

As explained in \cite[Part II. Chapter 2]{Jan03} and \cite[Section 8.1.2]{Hid04}, if $R$ is flat over $\Ss^0$, this is an $R$-model for the highest weight representation of $H_0$ with respect to $(T_{H_0}, B_{H_0}^\opp)$ of weight $\kappa$. 

Now, assume that $\square = \emptyset$ and hence, $R$ is a $\KK'$-algebra. Via the identification of $P_0$ and $P_h$ over $\CC$, $W_\kappa$ is a representation of $P_h$. As explained in \cite[Section 7.1]{Har86}, it therefore corresponds to an homogeneous $G$-vector bundle over $\check{X}$, the compact dual of $X$. The latter induces an automorphic vector bundle $\w_{W_\kappa}$ on $\level{K}{\Sh}$, for any $K$ as in Section \ref{can bundles}. As explained in \cite[Section 6.1.1]{EHLS}, it has a canonical model over some finite field extension $F(\kappa)$ of $F$ contained in $\KK'$. Its base change to $\KK'$ has a canonical extension to the toroidal compactification $\level{K}{\Sh^{\tor}_\Omega}$ of $\level{K}{\Sh}$, for any polyhedral cone decomposition $\Omega$.

Indeed, the restriction of
\[
    \w_\kappa = \w_{\kappa, \Omega} = s_{K, \Omega}^* \pi_* (\OO_{\EE}[\kappa]) \ ,
\]
where $s_{K, \Omega}$ is the canonical inclusion $\level{K}{\Sh^{\tor}_\Omega} \hookrightarrow \MM_{K,\Omega}^{\tor}$, to $\level{K}{\Sh}$ is canonically isomorphic to $\w_{W_\kappa}$. We denote both by $\w_\kappa$ when no confusion arises.

Furthermore, the subcanonical bundle of $\w_{W_\kappa}$ corresponds to the twist $\w_\kappa(-D_\Omega)$, where $D_\Omega$ is the ideal sheaf of the boundaries. In other words, it is the Cartier divisor $\level{K}{\Sh^{\tor}_\Omega} - \level{K}{\Sh}$ equipped with its structure of reduced closed subscheme.

The space of modular forms (for $G$) of weight $\kappa$ and level $K$ is
\[
    M_\kappa(K; R) 
    := 
        H^0(
            \level{K}{\Sh^\tor}_{/R},
            \w_\kappa
        )
    = \varinjlim_{\Omega} H^0(
            \level{K}{\Sh^\tor_\Omega}_{/R},
            \w_\kappa
        ) \ ,
\]
where the limit runs over all polyhedral cone decompostion $\Omega$, partially ordered via refinements. Similarly, the space of cusp forms $S_\kappa(K; R)$ is defined as
\[
    H^0(
            \level{K}{\Sh^\tor}_{/R},
            \w_\kappa^{\sub}
        )
    = \varinjlim_{\Omega} H^0(
            \level{K}{\Sh^\tor_\Omega}_{/R},
            \w_\kappa(-D_\Omega)
        ) \ .
\]

\subsubsection{$P$-nebentypus.} \label{P nebentypus}
In this chapter, we set $\square = \{p\}$, so let $S^0 := \Ss^0 = \OO_{\KK', (\p')}$. Fix an $S^0$-algebra $R \subset \CC$. Observe that the objects from the section above are all well-defined over $S^0_{\{p\}} = \OO_{\KK', (\p')}$ if we restrict our attention to level subgroups $K$ of the form $K = G(\ZZ_p) K^p$ or $K = K_r = I_rK^p$ for some $r \geq 1$.

As in section \ref{can bundles}, let $\tau$ be a smooth finite-dimensional representation of $L_P(\ZZ_p)$ that factors through $\LL_r = L_P(\ZZ_p/p^r\ZZ_p)$. Let $M_\tau$ denote the associated module over a finite ring extension of $\OO_F \otimes \ZZ_{(p)}$ contained in $\CC$. Enlarging the latter if necessary, we assume that it contains $S^0$ and denote it $S^0[\tau] \subset \CC$.

Define
\[
    \w_{\kappa, r, \tau} = 
        s_{K}^* (\pi_{r, \tau})_*
        (
            \OO_{\EE_{r, \tau}}[\kappa]
        )
\]
as a sheaf over $\level{K_r}{\ol{\Sh}}$. We denote its restriction to $\level{K_r}{\Sh}$ by $\w_{\kappa, r, \tau}$ as well.

\begin{definition}
    For any $S_0[\tau]$ algebra $R$, a modular form over $R$ on $G$ of weight $\kappa$, level $K_r$ and \emph{$P$-nebentypus} $\tau$ is a global section of $\w_{\kappa, r, \tau}$ over $\level{K_r}{\ol{\Sh}}$. The $R$-module of all such forms is denoted $M_{\kappa}(K_r, \tau; R)$.

    The $R$-module $S_{\kappa}(K_r, \tau; R)$ of cuspidal forms over $R$ on $G$ of weight $\kappa$, level $K_r$ and $P$-nebentypus $\tau$ is similarly defined by replacing $\w_{\kappa, r, \tau}$ with its twist by the ideal sheaf of the boundaries.
\end{definition}

A modular form $f \in M_{\kappa}(K_r, \tau; R)$ can be interpreted as a functorial rule that assigns to a tuple $(\ul{A}, \varepsilon, \phi) \in \EE_{r, \tau}(R')$, over an $R$-algebra $R'$, an element $f(\ul{A}, \epsilon, \phi) \in (M_\tau)_{/R'}$ such that
\[
    f(\ul{A}, b\epsilon, \phi \circ l) = 
        \kappa(b) 
        \tau(l) 
        f(\ul{A}, \epsilon, \phi)
\]
for all $b \in B_{H_0}(R')$ and $l \in L_P(\ZZ_p)$.

\begin{remark}
    Classically, the nebentypus of a modular form is a finite-order character of the maximal torus $T_H(\ZZ_p)$ of $H$. In our terminology, this is equivalent to a $B$-nebentypus.
\end{remark}

One similarly defines $\w_{\kappa, r}$ as the pullback to  $\level{K_r}{\ol{\Sh}}$ of $(\pi_r)_* \OO_{\EE_r}[\kappa]$ and $\w_{\kappa, r}^{\sub}$ as its twist by the ideal sheaf of the boundaries. Define
\[
    M_\kappa(K_r; R) 
    = 
        H^0(
            \level{K_r}{\ol{\Sh}}_{/R},
            \w_{r,\kappa}
        )
    \ \ \ \text{and} \ \ \
    S_\kappa(K_r; R) 
    = 
        H^0(
            \level{K_r}{\ol{\Sh}}_{/R},
            \w_{r,\kappa}^{\sub}
        ) \ .
\]

Since $L_P(\ZZ_p)$ is a compact group, one readily sees that
\[
    M_\kappa(K_r; R) =
        \bigoplus_{\tau}
            M_\kappa(K_r, \tau; R)
    \ \ \ \text{and} \ \ \
    S_\kappa(K_r; R) =
        \bigoplus_{\tau}
            S_\kappa(K_r, \tau; R)
\]
where the direct sum runs over all smooth irreducible representations over $R$ of $L_P(\ZZ_p)$ that factor through $\LL_r$.

\subsection{Weight types of (anti-)holomorphic automorphic representations.}
Let $G = G_1 = GU(V)$ be the unitary group (over $\ZZ$) associated to the PEL datum $\PP = \PP_1$. Recall that its signature is a collection of pairs of integers $\{(a_\sigma, b_\sigma)_{\sigma \in \Sigma_\KK}\}$.

Let $\square = \emptyset$ or $\{p\}$ and fix a neat open compact subgroup $K$ as in Section \ref{mod space and Shi var}. The dimension of $\level{K}{\Sh}(V)$ is equal to the $\CC$-dimension of $X_\PP$, namely
\[
    d = \sum_{\sigma \in \Sigma_\KK} a_\sigma b_\sigma \ .
\]

For any $i = 0, \ldots, d$, we write
\[
    H^i(\Sh(V), \w_\kappa) = 
    \varinjlim_K 
        H^i(
            \level{K}{\Sh^\tor}_{/R},
            \w_\kappa
        )
    \ \ \ \text{and} \ \ \
    H^i(\Sh(V), \w_\kappa^{\sub}) = 
    \varinjlim_K 
        H^i(
            \level{K}{\Sh^\tor}_{/R},
            \w_\kappa^{\sub}
        )
\]
and define
\[
    H^i_!(\Sh(V), \w_\kappa) = 
    \mathrm{Im}
        \left(
            H^i(\Sh(V), \w_\kappa^{\sub})
                \to
            H^i(\Sh(V), \w_\kappa)
        \right)
\]
as modules over $\Ss^0$.

\subsubsection{Comparison to $(\P_h, K_h)$-cohomology.} \label{comp cohomology theories}
In this section, we recall some of the results of \cite[Section 6.2]{EHLS} that are relevant for us later, especially in Section \ref{Pord mod forms}. 

We use the identification of $P_0$ (resp. $H_0$) and $P_h$ (resp. $C$) over $\CC$ without comments. Therefore, we identify modules equipped with actions from these groups (or their Lie algebra) repeatedly. Moreover, we write $K_h$ instead of $U_\infty$ for the real points of $C$.

Let $\g = \Lie(G(\RR))_\CC$. The adjoint action of $\mathrm{Ad}(h(\sqrt{-1}))$ induces the Harish-Chandra decomposition $\g = \p_h^- \oplus \k_h \oplus \p_h^+$. The Lie algebra of $P_h(\CC)$ is $\P_h = \p_h^- \oplus \k_h$. 

Therefore, for any dominant weight $\kappa$ of $T_{H_0}$, the highest weight representation $W_\kappa$ as a natural structure as a $(\P_h, K_h)$-module. Over $\CC$, there is a canonical isomorphism
\begin{equation} \label{H! vs H(b, K)}
    H^i_!(\Sh(V), \w_\kappa) 
        \xrightarrow{\sim} 
    H^i(
        \P_h, K_h;
        \Ab_0(G) \otimes W_\kappa
    ) 
    \ \ \ \ \ 
    (\text{for } i=0 \text{ and } d)
\end{equation}
of $G(\AA_f)$-modules, where $\Ab_0(G)$ is the space of cusp forms on $G$.

For any $\phi \in \Ab_0(G)$, let $\ol{\phi}(g) = \ol{\phi(g)}$. As explained in \cite[Section 6.2.1]{EHLS}, the map $\phi \mapsto \ol{\phi}$ induces a $c$-semilinear $G(\AA_f)$-equivariant isomorphism
\[
    c_B : 
        H_!^0(
            \P_h, K_h;
            \Ab_0(G) \otimes W_\kappa
        )
    \xrightarrow{\sim}
        H_!^d(
            \P_h, K_h;
            \Ab_0(G) \otimes W_{\kappa^D}
        ) \ .
\]

Here, $\kappa^D$ is again a dominant weight of $T_{H_0}$ (depending on $\kappa$ and the signature of $G$ at archimedean places) defined in \cite[Section 6.1.3]{EHLS} but whose exact formula is not relevant for us.

Let $\pi = \pi_\infty \otimes \pi_f$ be an irreducible $(\g, K_h) \times G(\AA_f)$-subrepresentation of $\Ab_0(G)$. From now on, we refer to such an object as a \emph{cuspidal automorphic representations} (without mentioning its irreducibility).
\begin{definition}
    Let $\pi$ and $\kappa$ be as above and $K$ be \emph{any} open compact subgroup of $G(\AA_f)$. We say that $\pi$ is \emph{holomorphic} of weight type $(\kappa, K)$ if
    \[
        \pi_f^K \neq 0
        \ \ \ \text{and} \ \ \
        H^0(
            \P_h, K_h;
            \pi_\infty \otimes W_\kappa
        )
    \]

    On the other hand, we say that $\pi$ is \emph{anti-holomorphic} of weight type $(\kappa, K)$ if
    \[
        \pi_f^K \neq 0
        \ \ \ \text{and} \ \ \
        H^d(
            \P_h, K_h;
            \pi_\infty \otimes W_{\kappa^D}
        ) \neq 0
    \]
\end{definition}

\begin{remark} \label{E pi rational}
    As explained in \cite{BHR94}, if $\pi$ is holomorphic or anti-holomorphic, then $\pi_f$ is defined over some number field $E(\pi)$. Enlarging it if necessary, we always assume it contains $\KK'$. 
\end{remark}

Let $\ol{\pi}$ be the image of $\pi$ via the $c$-semilinear map $\phi \mapsto \ol{\phi}$ on $\Ab_0(G)$. The isomorphism $c_B$ induces an involution $\pi \mapsto \ol{\pi}$ on the set of cuspidal automorphic representations of $G$. By definition, it interchanges holomorphic and anti-holomorphic representations but preserves weight type.

As explained in \cite[Section 6.5.3]{EHLS}, if $\pi$ has weight type $(\kappa, K)$, there is an isomorphism
\[
    \ol{\pi} \cong \pi^\vee \otimes ||\nu||^{a(\kappa)} =: \pi^\flat \ ,
\]
where $\nu$ is the similitude character on $G$ and
\[
    a(\kappa) = 2\kappa_0 + \sum_{\sigma \in \Sigma_{\KK}} \sum_{j=1}^{b_\sigma} \kappa_{\sigma, j} \ .
\]

In the next sections, we consider certain (anti-)holomorphic cuspidal automorphic representations $\pi$ of weight type $(\kappa, K)$ whose local factor at $p$ has a non-zero fixed $I_{P, r}$-vector for some $r \gg 0$. In that case, $\pi$ is of weight type $(\kappa, K_{P,r})$ for all $r \gg 0$. 

If the representation satisfies further conditions with respect to certain Hecke operators at $p$, we say that such $\pi$ is $P$-ordinary or $P$-anti-ordinary. We compare structures of $P$-ordinary and $P$-anti-ordinary representations using pairs of contragredient representations. Therefore, the involution $\pi \mapsto \pi^\flat$ is more convenient than $\pi \mapsto \ol{\pi}$ to analyze these dual notions.

\section{Structure theorem for $P$-ordinary representations.} \label{can Pord vectors}
In this section, we finally introduce the notion of ``$P$-ordinary'' holomorphic automorphic representations on $G = G_1$. The main results are Theorems \ref{holo canonical vector} and \ref{Pord structure thm}.

We obtain direct consequences for the dual notion of $P$-anti-ordinary vectors in the next section. Furthermore, all statements can be adapted for $G_2$, the opposite group of $G_1$ introduced in Section \ref{P1 P2 P3 P4}. We study the theory on $G_2$ more carefully in Section \ref{P ord and anti ord on G2}. 

\subsection{$P$-ordinary representations.} \label{Pord rep defn}
Given $w \in \Sigma_p$ and $1 \leq j \leq n$, let $t_{w,j} \in \GL_{n}(\OO_w)$ denote the diagonal matrix
\[
    t_{w,j} = 
    \begin{cases}
        \diag(p1_j, 1_{n - j}), &
        \text{if } j \leq a_{w} \\
        \diag(p1_{a_{w}}, 1_{n - j}, p1_{j - a_{w}}), &
        \text{if } j > a_{w}
    \end{cases}
\]

It corresponds to an element of $G(\QQ_p)$ under \eqref{prod G over Zp}, which we denote $t^+_{w,j}$ (namely, all its other components are equal to 1). Set
\[
    U_{w,j} = K_r t_{w,j}^+ K_r 
\]

We normalize these operators as follows. Fix an $S^0$-algebra $R \subset \CC$ as in Section \ref{P nebentypus}. Given a character 
$
    \kappa 
    = 
    (
        \kappa_0, 
        (\kappa_\sigma)_{\sigma \in \Sigma_\KK}
    )
$ of $T_{H_0}$ over $R$, let $\kappa_p$ be the character of $T_P(\ZZ_p) = T_H(\ZZ_p)$ such that
\[
    \kappa_p(t)
    =
    \prod_{w \mid p}
    \prod_{
        \substack{
            \sigma \in \Sigma_\KK \\ \p_\sigma = \p_w
        }
    }
    \prod_{j=1}^{a_w}
        \sigma(t_{w,j})^{\kappa_{\sigma c, j}} \ ,
\]
where $t = (\diag(t_{w,1}, \ldots, t_{w,a_w}))_{w \mid p}$ via \eqref{def H}. 

We also define the $T_{H_0}$-character
$
    \kappa_{\norm} 
    = 
    (
        \kappa_0, 
        (\kappa_{\norm, \sigma})_{\sigma \in \Sigma_\KK}
    )
$, where
\begin{equation}
    \kappa_{\norm, \sigma} 
    = 
    (
        \kappa_{\sigma, 1} - b_\sigma, 
        \ldots,
        \kappa_{\sigma, b_\sigma} - b_\sigma
    ) \ .
\end{equation}

Let $\kappa' = (\kappa_{\norm})_p$, viewed as a character of $T_P(\ZZ_p)$. Then, the $j$-th normalized Hecke operator at $p$ of weight $\kappa$ is defined as
\begin{equation} \label{norm u w j}
    u_{w,j} = u_{w,j,\kappa} 
    := 
        \absv{
            \kappa'(t_{w,j})
        }_p^{-1} 
        U_{w,j} \ .
\end{equation}

These operators can be interpreted as correspondences on the Igusa tower associated to $G$ (see \cite[Section 2.9.5]{EHLS}, \cite[Section 8.3.1]{Hid04} or \cite{SU02}) but this point of view will not be relevant for us in this article.

For $w \in \Sigma_p$, recall that we fixed partitions
\[
    \d_{w} = \left( n_{w, 1}, \ldots, n_{w, t_{w}} \right)
    \ \ \ \text{and} \ \ \
    \d_{\ol{w}} = \left( n_{\ol{w}, 1}, \ldots, n_{\ol{w}, t_{\ol{w}}} \right)
\]
of $a_{w}$ and $b_{w}$ in Section \ref{level at p}. Let $r_w = t_w + t_{\ol{w}}$ and consider
\[
    \widetilde{\d}_w = 
    \left( 
        \widetilde{\d}_{w,1}, 
        \ldots,
        \widetilde{\d}_{w, t_w};
        \widetilde{\d}_{w,t_w+1}, 
        \ldots,
        \widetilde{\d}_{w, r_w}
    \right) :=
    \left( 
        n_{w, 1}, 
        \ldots, 
        n_{w, t_w} ; 
        n_{\ol{w}, t_{\ol{w}}}, 
        \ldots, 
        n_{\ol{w}, 1} 
    \right) \ ,
\]
a partition of $n = a_w + b_w$. For $j=1, \ldots, r_w$, let $D_w(j)$ be the partial sum $\sum_{i=1}^j \widetilde{\d}_{w,i}$. 

Furthermore, set
\[
    u_{P, p} = u_{P, p,\kappa} :=
    \prod_{w \in \Sigma_p}
    \prod_{j = 1}^{r_w}
        u_{w,D_w(j),\kappa}
\]

\begin{definition}
    The \emph{$P$-ordinary projector of weight $\kappa$} as
    \[
        e_P = e_{P, \kappa} 
        := 
            \varinjlim \limits_{n} 
                u_{P, p, \kappa}^{n!}
    \]
\end{definition}

Let $\pi = \pi_\infty \otimes \pi_f$ be a holomorphic cuspidal automorphic representation of weight type $(\kappa, K_r)$ for some $r \geq 0$. The double coset operator $U_{w,D_w(j)}$ acts on $\pi_f^{K_r}$ via the action of $G(\AA_f)$ on $\pi_f$. In fact, writing
\[
    \pi_f = \pi_p \otimes \left( \bigotimes_{l \neq p} \pi_l \right) \ ,
\]
it acts as the double coset operator $U_{w, D_w(j), \kappa}^{\GL} := I_{P, r}t_{w,D_w(j)}I_{P,r}$ on $\pi_p^{I_r}$. It is well known that the generalized eigenvalues of $u_{w,j,\kappa}$ are $p$-adically integral. Therefore, the $P$-ordinary projector $e_P$ is well-defined as an operator on $\pi_f^{K_r}$ and $\pi_p^{I_r}$. 

\begin{definition}
    We say that $\pi$ is \emph{$P$-ordinary} (at $p$) of level $r \geq 0$ if its local factor $\pi_p$ contains a non-zero vector $\phi$ fixed by $I_r = I_{P,r}$ such that $e_P \phi = \phi$. The space $\pi_{p,r}^{\Pord} = e_P\pi_p^{I_{P,r}}$ is called the \emph{$P$-ordinary subspace} of $\pi_p$ (or of $\pi$) of level $r$. We say that its elements are the \emph{$P$-ordinary vectors} of $\pi_p$ of level $r$.
\end{definition}

If $\pi_p$ is $P$-ordinary of some level $r$, then it is $P$-ordinary of all levels $r \gg 0$. In particular, $\pi$ has weight type $(\kappa, K_r)$ for all $r \gg 0$.

\begin{remark}
    When $P = B$, a result of Hida (see \cite[Corollary 8.3]{Hid98} or \cite[Theorem 6.6.9]{EHLS}) implies that the space of $B$-ordinary vectors (or simply \emph{ordinary} vectors) is at most 1-dimensional and does not depend on $r$. This is no longer true for general parabolic subgroups $P$. However, Theorem \ref{Pord structure thm} yields an analogous result for $P$-ordinary subspaces.
\end{remark}

Clearly, $\phi \in \pi_p$ is $P$-ordinary if and only if $\phi \in \pi_p^{I_r}$, for all $r \gg 0$, such that $\phi$ is a simultaneous eigenvector for all operators $u_{w, D_w(j)}$ such that each eigenvalue is a $p$-adic unit.

Since $I_r$ is normal in $I_r^0$, the space $\pi_p^{I_r}$ is stable under the action of $I_r^0/I_r \cong \LL_r$. Let $\tau$ be an irreducible finite-dimensional smooth representation of $L_P(\ZZ_p)$ that factors through $\LL_r$. If a $P$-ordinary vector $\phi \in \pi_p^{I_r}$ lies in the $\tau$-isotypic component of $\pi_p^{I_r}$, we say that $\phi$ is  \emph{$(P, \tau)$-ordinary} or that it is \emph{$P$-ordinary of type $\tau$}. Let $\pi_{p,r}^{(P, \tau)}$ denote the subspace consisting of all $(P, \tau)$-ordinary vectors.

One readily sees that any $P$-ordinary vector is the finite sum of $(P, \tau)$-ordinary vectors for finitely many different representations $\tau$ as above. In particular,
\[
    \pi_{p,r}^{\Pord} = \bigoplus_{\tau} \pi_{p,r}^{(P, \tau)} \ ,
\]
as $\tau$ runs over all irreducible smooth representations of $L_P(\ZZ_p)$ that factor through $\LL_r$.

\begin{remark}
    In Definition \ref{def PIwahori}, one could replace $I_{P,r}$ with the collection of $g \in G(\ZZ_p)$ such that $g \mod p^r$ is in $(\ZZ_p/p^r\ZZ_p)^\times \times SP(\ZZ_p/p^r\ZZ_p)$. Here, $SP$ is the derived subgroup of $P$ or equivalently, it is the product in $P$ over $w \in \Sigma_p$ of the subgroups $SP_w \subset P_w$ consisting of upper-block triangular matrices whose diagonal blocks all have determinant 1. Let us write the corresponding group by $I_{SP, r}$ momentarily, in which case we have $I_{P, r} \subset I_{SP, r} \subset I_{P,r}^0$.
    
    Then, one can define $P$-ordinary representations of $G$ using $I_{SP, r}$ instead of $I_{P,r}$. By doing so, the space of $P$-ordinary vectors decomposes a direct sum over all $P$-nebentypus of $\tau$ that factor through $\det : L_P(\ZZ_p) \to \ZZ_p^\times$. Doing so is obviously less general but has the advantage of simplify the theory as only characters of $L_P(\ZZ_p)$ occur as types of $P$-ordinary vectors. On the other hand, systematically developing the more general theory (with $P^u$ instead of $SP$) has the advantage that any holomorphic cuspidal representation $\pi$ of $G$ is trivially $\GL(n)$-ordinary. We discussed our motivation to study this more general notion in the introduction of this paper.
\end{remark}

\subsection{Local factors at places $w \mid p$.}
The identifications \eqref{prod G over Zp} and \eqref{GL(ei Lw) basis} induce the isomorphism
\begin{equation} \label{facto G(Qp)}
    G(\QQ_p) \xrightarrow{\sim} \QQ_p^\times \times \prod_{w \in \Sigma_p} G_w \ ,
\end{equation}
where $G_w = \GL_n(\KK_w)$. 

Consider the groups $I_{w,r}, I_{w,r}^0, P_w \subset G_w$  constructed in Section \ref{level at p}. Recall that the decompositions \eqref{prod G over Zp} and \eqref{GL(ei Lw) basis} yield identifications
\begin{align*}
    P \xrightarrow{\sim} 
        \prod_{w \in \Sigma_p} P_w 
    \ \ \ ; \ \ \
    I^0_r \xrightarrow{\sim} 
        \ZZ_p^\times 
        \times \prod_{w \in \Sigma_p} I_{w, r}^0 
    \ \ \ ; \ \ \
    I_r \xrightarrow{\sim} 
        \ZZ_p^\times \times 
        \prod_{w \in \Sigma_p} I_{w, r}
\end{align*}

Let $\pi$ be a holomorphic cuspidal automorphic representation of $G(\AA)$ of type $(\kappa, K_r)$. Recall that the character $\kappa$ of $T_{H_0}$ is identifies with a tuple $(\kappa_0, (\kappa_\sigma)_{\sigma \in \Sigma_\KK})$ such that $\kappa_0 \in \ZZ$ and $\kappa_\sigma \in \ZZ^{b_\sigma}$. 

The above discussion allows one to factor the $p$-component $\pi_p$ of $\pi$ as
\begin{equation} \label{facto pi p}
    \pi_p \cong \mu_p \otimes \bigotimes_{w \in \Sigma_p} \pi_w \ ,
\end{equation}
where $\mu_p$ is a character of $\QQ_p^\times$ and $\pi_w$ is an irreducible admissible representation of $G_w$.

Let $u_{w, D_w(j), \kappa}^{\GL} := |\kappa'(t_{w,j})|^{-1}_p U_{w, D_w(j), \kappa}^{\GL}$, where $\kappa'$ related to $\kappa$ as in equation \eqref{norm u w j}. Then, the Hecke operators $u_{w, D_w(j), \kappa}$ from Section \ref{Pord rep defn} act on
\[
    \pi_p^{I_r} \cong (\mu_p)^{\ZZ_p^\times} \otimes \bigotimes_{w \in \Sigma_p} \pi_w^{I_{w,r}} \ .
\]
via the action of $u_{w, D_w(j), \kappa}^{\GL}$ on $\pi_w^{I_{w,r}}$. Again, this action is compatible as $r$ increases, hence we do not include it in the notation of the operator and the generalized eigenvalues of $u^{\GL}_{w, D_w(j), \kappa}$ are all $p$-adically integral.

For the remainder of Section \ref{can Pord vectors}, we assume that $\pi$ is $P$-ordinary and that
\begin{equation} \label{ineq kappa sigma}
    \kappa_{\sigma, b_\sigma} + \kappa_{\sigma c, a_\sigma} \geq n, \forall \sigma \in \Sigma_\KK \ .
\end{equation}

The fact that $\pi$ is $P$-ordinary is equivalent to $\mu_p$ being unramified and that, for each $w \in \Sigma_p$ and $r \gg 0$, there exists some non-zero $\phi \in \pi_w^{I_{w,r}}$ such that 
\[
    u_{w, D_w(j), \kappa}^{\GL} \phi 
    = 
        c_{w, D_w(j)} \phi \ ,
\]
where is $c_{w, D_w(j)}$ a $p$-adic unit, for all $1 \leq j \leq r_w$. 

In that case, we say that $\pi_w$ is \emph{$P_w$-ordinary} and that such a vector $\phi$ is $P_w$-ordinary (of level $r$). We denote the subspace of all $P_w$-ordinary vectors as $\pi_w^{\Pword}$. 
Note that $\phi \in \pi_w^{I_r}$ is $P_w$-ordinary if and only if $e_w \phi = \phi$, where $e_w$ is the \emph{$P_w$-ordinary projector}
\[
    e_w = 
        \lim_{n \to \infty} 
        \left(
            \prod_{j=1}^{r_w} 
            u_{w, D_w(j)}^{n!}
        \right) \ ,
\]
which has a well-defined action on $\pi_w^{I_{w,r}}$.

\subsubsection{Explicit computations.}
To clarify arguments in later proofs, we now describe explicit left coset representatives for $U_{w,D_w(j)}^{\GL}$. For simplicity, we only compute the left coset representatives when $j \leq t_w$. The same conclusion applies for $j > t_w$ but writing down the matrices is simply more cumbersome. In any case, fix $j \leq t_w$ and write $i = D_w(j)$ (making the dependence on $j$ implicit). 

Fix a uniformizer $\varpi \in \p_w$. Given any matrix $X \in I_{w,r}$, write it as
\[
	X = \begin{pmatrix}
		A & B  \\
		\varpi^r C & D
	\end{pmatrix}
\]
where $A \in \GL_{i}(\OO_w)$, $D \in \GL_{i}(\OO_w)$ and $B \in M_{i \times (n-i)}(\OO_w)$ and $C \in M_{(n-i) \times i}(\OO_w)$.

Fix a set $S_w$ of representatives in $\OO_w$ for $\OO_w/p\OO_w$. Let $B', B'' \in M_{i \times (n-i)}(\OO_w)$ be the unique matrices such that $B'$ has entries in $S_w$ and $BD^{-1} = B' + p B''$. Then, we have
\[
	X = 
	\begin{pmatrix}
		1_j & B' \\
		0 & 1_{n-j}
	\end{pmatrix}
	\begin{pmatrix}
		A - \varpi^r B'C & p B''D \\
		\varpi^r C' & D
	\end{pmatrix} =: X' X''
\] 

In particular, $t_{w,i}^{-1}X''t_{w,i}$ is in $I_{w,r}$. Therefore,
\[
	I_{w,r} t_{w,i} I_{w,r} = \bigsqcup_{x \in M_j} xt_{w,i} I_{w,r}
\]
where $M_j \subset \GL_n(\KK_w)$ is the subset of matrices $\begin{pmatrix} 1_{i} & B \\ 0 & 1_{n-i}\end{pmatrix}$ such that the entries of $B$ are in $S_w$.

In particular, this set of representative does not depend on $r$ and one obtains the same result by replacing $I_{w,r}$ with $N_w = \cap_r I_{w,r} = P_w^u(\KK_w) \cap \GL_n(\OO_w)$. One readily convinces themselves that the calculations above still apply for $t_w < j \leq r_w$.

Let $V_w$ be the $\KK_w$-vector space associated to $\pi_w$. By continuity, its $N_w$-invariant subspace $V_w^{N_w}$ is equal to $\cup_r V_w^{I_{w,r}}$. 

\begin{lemma} \label{lma U_j}
	There is a decomposition $V_w^{N_w} = V^{N_w}_{w, \inv} \oplus V^{N_w}_{w, \nil}$ such that, for $1 \leq j \leq r_w$, $U_{w, D_w(j)}^{\GL}$ is invertible on $V^{N_w}_{w, \inv}$ and nilpotent on $V^{N_w}_{w, \nil}$. Moreover, $U_{w, D_w(j)}^{\GL} = I_{w,r} t_{w, D_w(j)} I_{w,r}$ acts as $\delta_{P_w}(t_{D_w(j)})^{-1} t_{D_w(j)}$ on $V_{w, \inv}^{N_w}$.
\end{lemma}

\begin{proof}
    We keep writing $i = D_w(j)$ in this proof and omit the subscript $w$ in what follows. 
    
    The first part is a consequence of the explanations in \cite[Section 5.2]{Hid98}. Moreover, \cite[Proposition 5.1]{Hid98} shows that the natural projection from $V$ to its $P$-Jacquet module $V_P$ induces an isomorphism $V_{\inv}^N \cong V_P$ that is equivariant for the action of all the $U^{\GL}_i$ operators.
    
    From our explicit computations above, it is clear that $U^{\GL}_i$ acts on $V_P$ via $|M_j|t_{i}$, where $|M_j|$ is the cardinality of $M_j$. To see this, simply note that given any $x \in M_j$, $t_{i}^{-1} x t_{i} \in P_w^u(\KK_w)$ fixes $V_P$. Therefore, the result follows since $M_j$ contains exactly $|p|_w^{-i(n-i)} = \delta_P(t_i)^{-1}$ elements.
\end{proof}

It is clear from Lemma \ref{lma U_j} that any $P_w$-ordinary vector $\phi \in V_w^{N_w}$ lies in $V^{N_w}_{w,\inv}$ and 
\begin{equation} \label{s_j eigenvalue}
    \pi_w(t_{w, D_w(j)})(\phi) = 
        |\kappa'(t_{w, D_w(j)})|_p 
        \delta_{P_w}(t_{w, D_w(j)}) 
        c_{w, D_w(j)} \phi \ ,
\end{equation}
where $c_{w, D_w(j)}$ is its $u^{\GL}_{w,D_w(j)}$-eigenvalue (a $p$-adic unit). In particular, $\phi$ is a simultaneous eigenvector under the action of $\pi_w$ for all matrices $t_{w,D_w(j)}$.

\subsubsection{Bernstein-Zelevinsky geometric lemma for $P_w$-ordinary representations.} 
In Section \ref{BK types of Pord rep}, we obtain results about the structure of the $P_w$-ordinary subspace of $\pi_w$ via its relation to its $P_w$-Jacquet module, see the proof of Lemma \ref{lma U_j}. To understand further the $P_w$-Jacquet module of $\pi_w$, we use a version of the Bernstein-Zelevinsky geometric lemma (see \cite[Section VI.5.1]{Ren10} or \cite[Theorem 6.3.5]{Cas95}) that is adapted to our setting, see Lemma \ref{P-Jacquet filtration}. However, we first need to introduce some notation.

\begin{lemma}\label{Jacquet Lemma}
    Let $\pi_w$ be a $P_w$-ordinary representation of $G_w$. There exists a parabolic subgroup $Q_w \subset P_w$ of $G_w$ and a supercuspidal representation $\sigma_w$ of $Q$ such that $\pi_w \subset \ind{Q_w}{G_w} \sigma_w$.
\end{lemma}

\begin{proof}
    The following is a minor modification of the proof of Jacquet's theorem \cite[Theorem 5.1.2]{Cas95}. Moreover, we omit the subscript $w$ to lighten the notation.
	
    The fact that $\pi$ is $P$-ordinary implies that $\r{P}{G} \pi \neq 0$. By \cite[Theorem 3.3.1]{Cas95}, the latter is both admissible and finitely generated so it admits an irreducible admissible quotient $\tau$ as a representation of $L$. 
	
    By Frobenius reciprocity \cite[Theorem 2.4.1]{Cas95} and the irreducibility of $\pi$, it follows that $\pi \subset \ind{P}{G} \tau$. Then, it is a theorem of Jacquet \cite[Theorem 5.1.2]{Cas95} that there exists a parabolic $Q_L \subset L$ and a supercuspidal representation $\sigma$ of its Levi factor such that $\tau \subset \ind{Q_L}{L} \sigma$. By transitivity of parabolic induction, the result follows.
\end{proof}

Fix an embedding $\pi_w \hookrightarrow \ind{Q_w}{G_w} \sigma_w$ with the notation as in Lemma \ref{Jacquet Lemma}. Let $M_w$ and $Q_w^u$ denote the Levi factor and unipotent radical of $Q_w$. 

Moreover, let $B_w$ denote the Borel subgroup of $G_w$ corresponding to the trivial partitions, as in Remark \ref{trivial partition}. Let $T_w$ denote the Levi factor of $B_w$. In particular, $T_w$ is the maximal torus of $G_w$.

Let $W$ be the Weil group of $G_w$ with respect to $(B_w, T_w)$ and consider
\[
    W(P_w, Q_w) = 
        \{x \in W \mid 
            x^{-1}(L_w \cap B_w)x
                \subset B_w, 
            x(M_w \cap B_w)x^{-1}  
                \subset B_w
        \} \ .
\]

According to \cite[Section V.4.7]{Ren10}, for each $x \in W(P_w, Q_w)$,  $xP_wx^{-1} \cap M_w$ is a parabolic subgroup of $M_w$ with Levi factor equal to $xL_wx^{-1} \cap M_w$. Similarly, the Levi factor of the parabolic subgroup $L_w \cap x^{-1}Q_wx \subset L_w$ is $L_w \cap x^{-1}M_wx$.

Denote the natural \emph{conjugation-by-$x$} functor that sends a representation of $xLx^{-1} \cap M_w$ to a representation of $L_w \cap x^{-1}M_wx$ by $(\cdot)^x$. Moreover, let $W(L_w, M_w)$ be the subset of $x \in W(P_w, Q_w)$ such that $xL_wx^{-1} \cap M_w = M_w$, and so $L_w \cap x^{-1}M_wx = x^{-1}M_wx$. Note that this does not imply that $L_w \cap x^{-1}Q_wx$ is equal to $x^{-1}Q_wx$ but rather that its Levi subgroup is $x^{-1}M_wx$.

The following is a version of \cite[Theorem 6.3.5]{Cas95} that is adapted to our setting and notation.

\begin{lemma} \label{P-Jacquet filtration}
    Let $Q_w \subset P_w$ denote standard parabolic subgroups of $G_w$ as above and let $\sigma_w$ be an irreducible supercuspidal representation of $M_w$.
    
    There exists a filtration, indexed by $W(L_w, M_w)$, of the $L_w$-representation $\r{P_w}{G_w} \ind{Q_w}{G_w} \sigma_w$ such that the subquotient corresponding to $x \in W_{L_w}$ is isomorphic to $\ind{L_w \cap x^{-1}Q_wx}{L_w} \sigma_w^x$  and the one corresponding to $x=1$ is a subrepresentation.
\end{lemma}

\begin{proof}
    In this proof, we drop the subscript $w$ to lighten the notation.
    
    The Bernstein-Zelevinsky geometric lemma (see \cite[Section VI.5.1]{Ren10}) states that there exits a filtration of $\r{P}{G} \ind{Q}{G} \sigma$ such that the corresponding graded pieces are isomorphic to
    \[
	\ind{L \cap x^{-1}Qx}{L} \left( \r{xPx^{-1} \cap M}{M} \sigma \right)^x
    \]
    as $x$ runs over all elements of $W(P,Q)$. Moreover, one can order the filtration so that the factor corresponding to $\sigma$ (i.e. the graded piece corresponding to $x = 1$) is a subrepresentation of $\r{P}{G} \ind{Q}{G} \sigma$.

    Since $\sigma$ is supercuspidal, the graded piece corresponding to $x \in W(P,Q)$ is nonzero if and only if $xLx^{-1} \cap M = M$, i.e. $x \in W(L, M)$. For such an $x$, the graded piece is clearly isomorphic to $\ind{L \cap x^{-1}Qx}{L} \sigma^x$.
\end{proof}

\subsection{Main Theorems.} \label{BK types of Pord rep}
For simplicity, we assume that $\pi_p$ satisfies the following hypothesis :
\begin{hypothesis} \label{Qw equals Pw}
    The parabolic subgroup $Q_w$ for $\pi_w$ from Lemma \ref{Jacquet Lemma} is equal to $P_w$ for all $w \in \Sigma_p$. In particular $\sigma_w$ is a supercuspidal representation of $L_w$.
\end{hypothesis}

\begin{remark} \label{sc support assumption}
    This hypothesis is certainly restrictive in our context. For instance, if $\pi_p$ is $B$-ordinary, then Lemma \ref{Jacquet Lemma} implies that all local factors $\pi_w$ lie in a principal series. Furthermore, if $\pi_p$ is $B$-ordinary (i.e. \emph{ordinary} in the usual sense) then it follows immediately from our definitions that it is also $P$-ordinary. Therefore, the case $Q_w \neq P_w$ can certainly occurs.
    
    One can argue that this is not a major issue since in the situation above, if $\pi_p$ is $B$-ordinary than there is little interest in considering its structure as a $P$-ordinary representation. One only obtains less information this way. However, if $\pi_p$ is a general $P$-ordinary representation whose local factors $\pi_w$ lie in a principal series, it is not necessarily true that $\pi_p$ is also $B$-ordinary. In general, if $\pi_p$ is $P$-ordinary and the supercuspidal support of all $\pi_w$ is $Q_w$, then $\pi_p$ might not be $Q$-ordinary, where $Q = \prod_w Q_w$. Therefore, the hypothesis above restricts us to study certain $P$-ordinary representations that are not $Q$-ordinary with respect to any smaller parabolic $B \subset Q \subsetneq P$.
    
    In subsequent work, the author plans to generalize the theory and results below for any parabolic subgroup $Q_w \subset P_w$ using the theory of \emph{covers} of types developed in \cite{BusKut98, BusKut99} and \emph{typical representations} as in \cite{Lat21}.
\end{remark}

\begin{theorem} \label{holo canonical vector}
    Let $\pi$ be a $P$-ordinary representation as above such that its weight $\kappa$ satisfies Inequality \eqref{ineq kappa sigma}. Let $\pi_w \subset \ind{P_w}{G_w} \sigma_w$ be its component at $w \in \Sigma_p$ as above, a $P_w$-ordinary representation.
    \begin{enumerate}
        \item[(i)] For $r \gg 0$, let $\phi, \phi' \in \pi_w^{I_r}$ be $P_w$-ordinary vectors. Let $\varphi$ and $\varphi'$ be their respective image in $\ind{P_w}{G_w} \sigma_w$. If $\phi \neq \phi'$, then $\varphi(1) \neq \varphi'(1)$.
        \item[(ii)] For $r \gg 0$, let $\phi \in \pi_w^{I_r}$ be a simultaneous eigenvector for the $u_{w, D_w(j)}$-operators that is not $P_w$-ordinary. Let $\varphi$ be its image in $\ind{P_w}{G_w} \sigma_w$. Then, $\varphi(1) = 0$.
        \item[(iii)] Let $\tau_w$ be a smooth irreducible representation of $L_w(\OO_w)$. Assume there exists an embedding $\tau_w \hookrightarrow \sigma_w$ over $L_w(\OO_w)$. Let $X_w$ be the vector space associated to $\tau_w$, viewed as a subspace of the one associated to $\sigma_w$. 
        
        Then, given $\alpha \in X_w$, there exists some $r \gg 0$ such that $\tau_w$ factors through $L_w(\OO_w/\p_w^r\OO_w)$ and some (necessarily unique) $P_w$-ordinary $\phi_{r,\alpha} \in \pi_w^{I_r}$ such that $\varphi_{r,\alpha}(1) = \alpha$, where $\varphi_{r,\alpha}$ is the image of $\phi_{r,\alpha}$ in $\ind{P_w}{G_w} \sigma_w$. Furthermore, the support of $\varphi_{r,\alpha}$ contains $P_w I_{w,r}$. The map $\alpha \mapsto \phi_{r,\alpha}$ yields an embedding of $L_w(\OO_w)$-representations 
        \[
            \tau_w \hookrightarrow \pi_{w,r}^{\Pword} \ .
        \]
    \end{enumerate}
\end{theorem}

\begin{proof}
    This proof is inspired by the one of \cite[Lemma 8.3.2]{EHLS} which is itself inspired by arguments in \cite[Section 5]{Hid98}. By abuse of notation, we will always write $L$ when we mean $L(\KK_w)$. However, we still write $L(\OO_w)$ when referring to its maximal compact subgroup. From now on, we omit the subscript $w$ in this proof.
	
    As explained above, the space of $P$-ordinary vector is contained in $V^N_{\inv}$ and $\pr_P : V \to V_P$ induces an isomorphism on $V^N_{\inv} \xrightarrow{\sim} V_P$ which is equivariant for the action of $L(\OO)$ and the $u^{\GL}_{D(j)}$-operators. Let $s_P : V_P \to V^N_{\inv}$ denote its inverse.

    Consider the natural inclusion $V \hookrightarrow \ind{P}{G} \sigma$ and the corresponding embedding $V_P \hookrightarrow (\ind{P}{G} \sigma)_P$ as representations of $L$, using the fact that the $P$-Jacquet module functor is exact. Note here that we are using the unnormalized version of the $P$-Jacquet functor.

    Consider the filtration indexed by $W(L, L)$ of $(\ind{P}{G} \sigma)_P$ from Lemma \ref{P-Jacquet filtration}. We use a version with unnormalized $P$-Jacquet functor, hence the graded piece corresponding to $x \in W(L, L)$ is isomorphic to $\sigma^x \delta_P^{1/2}$.
 
    First, we claim that $\pr_P$ maps any simultaneous eigenvector for the $u_{D(j)}$-operators whose eigenvalues are all $p$-adic units inside that subrepresentation.

    One readily checks that $x \in W(P,P)$ is in $W(L,L)$ if and only if it simply permutes the $\GL_{n_k}(\KK_w)$-blocks of $L$ of the same size. In particular, exactly one such $x \in W(L,L)$ acts trivially on the center $Z(L)$ of $L$, namely $x = 1$, while any other $1 \neq x \in W(L,L)$ stabilizes but acts non-trivially on $Z(L)$.
	
    The operator $u^{\GL}_{D(j)}$ acts on $\sigma^x \delta_P^{1/2}$ via multiplication by
    \[
        \beta_x(s_j) = 
            \absv{\kappa' (s_j)}^{-1}_p 
            \delta_P^{-1/2}(s_j) 
            \omega_{\sigma}^x(s_j)
    \]
    where $s_j = t_{D(j)}$ and $\omega_{\sigma} : Z(L) \to \CC^\times$ is the central character of $\sigma$.

    These $\beta_x$ define characters of $Z(L)$. The $P$-ordinarity assumption implies that $\beta_1(s_j)$ is a $p$-adic unit for all $1 \leq j \leq t + r$ and therefore $\beta_1(s)$ is a $p$-adic unit for all $s \in Z(L)$. We claim that given any $x \in W(L, L)$, the values of $\beta_x$ on $Z(L)$ are all $p$-adic units if and only if $x = 1$.

    By recalling that $\delta_P$ and $\delta_B$ agree on $Z(L)$ and proceeding exactly as in the proof of \cite[Lemma 8.3.2]{EHLS}, one uses Inequality \eqref{ineq kappa sigma} to show that 
    \[
	\theta = |\kappa'|^{-1} \delta_P^{-1/2}
    \]
    is a regular character of $Z(L)$ and $\beta_x$ satisfies the above property if and only if $\theta^x = \theta$. By regularity, this only occurs when $x = 1$.

    The argument above shows that under the natural map
    \begin{equation} \label{V N inv to rho delta_P}
        V_{\inv}^N \hookrightarrow 
            V \twoheadrightarrow 
                V_P \hookrightarrow 
                    (\ind{P}{G} \sigma)_P \ ,
    \end{equation}
    the subspace of $P$-ordinary vector of $V$ injects into the subrepresentation $\sigma \delta_P^{1/2}$ of $(\ind{P}{G} \sigma)_P$.

    This map is exactly the composition of $s_P : V_{\inv}^N \xrightarrow{\sim} V_P$ with the map $i : V_P \to \sigma \delta_P^{1/2}$ corresponding under the Frobenius reciprocity to the inclusion $v \mapsto f_v$ of $V$ into $\ind{P}{G} \sigma$. In other words, this map is $v \mapsto f_v(1)$. Therefore, a $P$-ordinary vector $v \in V^N$ is uniquely determined by $f_v(1)$. This shows part (i). 

    For part (ii), pick a simultaneous eigenvector $v \in V^N_{\inv}$ for the $u^{\GL}_{D(j)}$-operators that is not $P$-ordinary. Then, as above, the map $i \circ s_P : V^N_{\inv} \to \sigma \delta_P^{1/2}$ maps $v$ to $f_v(1)$. By equivariance of the action of the $u^{\GL}_{D(j)}$-operators on both sides, we must have $f_v(1) = 0$.

    To show part (iii), consider $\alpha$ as an element of the vector space associated to $\sigma$, which is also the one associated to $\sigma \delta_P^{1/2} \subset V_P$. Let $\phi = s_P(\alpha) \in V_{\inv}^N$. In particular, $\phi \in \pi^{I_r}$ for some $r \gg 0$. We may assume that $r$ is sufficiently large so that $\tau$ factors through $L(\OO/\p^r\OO)$.

    Finally, since $\pr_P$ is equivariant under the action of the $u^{\GL}_{D(j)}$-operators and these act on $\pr_P(\phi) = \alpha$ via multiplication by the $p$-adic unit $\beta(s_j)$, one concludes that $\phi$ is $P$-ordinary. Proceeding as in the proof of part (i), we obtain $\varphi(1) = \pr_P \phi = \alpha$, where $\varphi \in \ind{P}{G} \sigma$ is the function corresponding to $\phi$. 

    Therefore, $\phi_{r,\alpha} := \phi$ is the desired vector, necessarily unique by part (i). The last statement holds because $s_P$ is $L(\OO_w)$-equivariant.
\end{proof}

\begin{remark}
    As a consequence of the proof for part (i) above, we see that $\pi_w$ is $P_w$-ordinary  (of level $r \gg 0$) if and only if
    \begin{equation} \label{central character condition}
        \beta(s) = 
            \absv{\kappa'(s)}_p^{-1}
            \delta_{P_w}^{-1/2}(s)
            \w_\sigma(s)
    \end{equation}
    is a $p$-adic unit for all $s \in Z(L_w(\KK_w))$. In other words, not all supercuspidal representation $\sigma_w$ can occur. Furthermore, when $\pi_w$ is $P_w$-ordinary (of level $r \gg 0$), the $u_{w, D_w(j), \kappa}^{\GL}$-eigenvalue of all the $P_w$-ordinary vectors is $\beta(t_{w, D_w(j)})$.
\end{remark}

\begin{remark} \label{covers of types to P-Iwahori}
    We now view $\tau_w$ as as a representation of $I_{w, r}^0$ via the identity $I_{w,r}^0 / I_{w,r} = L_w(\OO_w/\P_w^r\OO_w)$. Clearly, the embedding constructed in (iii) of Theorem \ref{holo canonical vector} is an embedding of $I_{w,r}^0$.  This shows that $\pi_w$ contains a cover of $\tau_w$ from $L_w$ to $\GL_n(\OO_w)$, in the sense of \cite{BusKut98, BusKut99}, in its subspace of $P_w$-ordinary vectors. In fact, the above theorem shows that this cover is exactly the space of $P_w$-ordinary vectors of type $\tau_w$, i.e. the $\tau_w$-isotypic subspace $\pi_{w,r}^{\Pword}[\tau_w]$ for all $r \gg 0$.
\end{remark}

Consider the BK-type $(L_w(\OO_w), \tau_w)$ of the supercuspidal representation $\sigma_w$, as defined in Section \ref{sc support}. Let $\tau$ be the representation of $L_P(\ZZ_p)$ corresponding to $\otimes_{w \in \Sigma_p} \tau_w$ under the natural isomorphism $L_P = \prod_{w \in \Sigma p} L_w$ induced by the identification \eqref{facto G(Qp)}. We refer to $\tau$ as the \emph{BK-type of $\pi_p$.}

\begin{theorem} \label{Pord structure thm}
    Let $\pi$ be a holomorphic $P$-ordinary representation of weight type $(\kappa, K)$ such that Inequality \eqref{ineq kappa sigma} holds. Let $\tau$ be the BK-type of $\pi_p$. Then,
    \[
        \hom_{L_P(\ZZ_p)}
        (
            \tau, 
            \pi_{p, r}^{\Pord}
        )
    \]
    is 1-dimensional for all $r \gg 0$. In other words, the space $\pi_{p}^{(P, \tau)} = \pi_{p,r}^{(P, \tau)}$ of $P$-ordinary vectors of type $\tau$ is independent of $r \gg 0$ and
    \[
        \dim 
        \left(
            \pi_p^{(P, \tau)}
        \right) 
        = \dim \tau
    \]
\end{theorem}

\begin{proof}
    Fix $w \in \Sigma_p$ and consider $\pi_{w,r}^{\Pword} = e_w \pi_w^{I_{w,r}}$ for $r \gg 0$. By Theorem \ref{holo canonical vector} (iii), there is a natural isomorphism
    \[
        \hom_{L_w(\OO_w)}(\uptau_w, \sigma_w) = \hom_{L_w(\OO_w)}(\uptau_w, \pi_{w,r}^{\Pword}[\uptau_w]) \ ,
    \]
    where $\uptau_w$ is any smooth irreducible representation of $L_w(\OO_w)$. From \cite[Proposition 5.6]{BusKut98}, we know that the BK-type $\tau_w$ of $\pi_w$ has multiplicity one in $\sigma_w$. Therefore, the result follows by applying the above to $\uptau_w = \tau_w$. 
\end{proof}

\section{$P$-anti-ordinary representations and opposite unitary groups.} \label{P-a-ord thms}

In this section, we first define the dual notion of \emph{$P$-anti-ordinary} representations and analyze the structure of $P$-anti-ordinary subspaces using our results above. Then, we again follow the material of \cite[Section 6.2]{EHLS} to compare the $P$-(anti-)ordinary representations on $G = G_1$ and its opposite unitary group $G_2$

The results on $G_1$ directly apply to $G_2$ simply by replacing $P$ with its opposite parabolic $P^{\opp}$. However, using standard intertwining operators, one obtains results with respect to $P$ once more. Our results are greatly inspired by \cite[Sections 8.3-8.4]{EHLS}.

Moreover, as explained in the introduction of this paper, our motivation is to use the results here for explicit calculations of zeta integrals in upcoming work of the author.

\subsection{$P$-anti-ordinary representations on $G_1$} \label{a.holo and a.ord on G1}
Let $\pi$ be an anti-holomorphic cuspidal representation on $G = G_1$ of weight type $(\kappa, K_r)$. For each $w \in \Sigma_p$ and $1 \leq j \leq n$, let $t_{w,j}^- = t_{w,j}^{-1} \in G(\QQ_p)$, where $t_{w,j}$ is the element constructed in Section \ref{Pord rep defn}. Proceeding as in that section, we define
\[
    U_{w,j}^- \ \ \ ; \ \ \ u_{w,j, \kappa}^- \ \ \ ; \ \ \ u_{P,p,\kappa}^- \ \ \ ; \ \ \ e_{P, \kappa}^-
\]
by replacing $t_{w,j}^+$ by $t_{w,j}^-$ in the definitions of $U_{w,j}$, $u_{w,j,\kappa}$, $u_{P,p,\kappa}$ and $e_{P, \kappa}$. We also consider the partition $\wt{\d}_w$ of $n$ of length $r_w$ as well as its partial sums $D_w(j)$ for $1 \leq j \leq r_w$.

As in Section \ref{Pord rep defn}, the generalized eigenvalues of the action of $u_{w, D_w(j), \kappa}^-$ on $\pi_f^{K_r}$ are all $p$-adically integral. Therefore, the \emph{$P$-anti-ordinary projector} $e_{P, \kappa}^-$ has a well-defined action on $\pi_f^{K_r}$. We say that $\pi$ is $P$-anti-ordinary (of level $r$) if $e_{P, \kappa}^- (\pi_f^{K_r}) \neq 0$.

\begin{remark}
    Note that the action of $U_{w,j}^-$ (and therefore all the other operators as well) does depend on $r$. However, by abuse of notation, we do not include $r$ in the already long list of subscripts.
\end{remark}

By definition, for each $w \in \Sigma_p$, $1 \leq j \leq r_w$ and $r \geq 0$, the operator $U_{w,j}^-$ acts on 
\[
    \pi_f^{K_r} = \pi_p^{I_r} \otimes \left( \bigotimes_{l \neq p} \pi_l \right)^{K^p}
\]
via its action on $\pi_p^{I_r}$. Furthermore, by writing $\pi_p \cong \mu_p \otimes \bigotimes_{w \in \Sigma_p} \pi_w$ using isomorphism \eqref{facto pi p}, its action on $\pi_p^{I_r} = \bigotimes_{w \in \Sigma_p} \pi_w^{I_w,r}$ is induced by the action of the double coset operator $U_{w, D_w(j)}^{\GL, -} = I_{w,r}t_{w,r}^-I_{w,r}$ on $\pi_w^{I_{w,r}}$.

Let $u_{w, D_w(j)}^{\GL, -} = |\kappa'(t_{w,j})|_pU_{w, D_w(j)}^{\GL, -}$, where $\kappa'$ related to $\kappa$ as in equation \eqref{norm u w j}, and
\[
    e_w^- = 
    \lim_{m \to \infty} 
    \left( 
        \prod_{w \in \Sigma_p} 
        \prod_{j=1}^{r_w}
            u_{w, D_w(j), \kappa}^{\GL, -}
    \right)^{m!}
\]

It follows from the discussion above that the generalized eigenvalues of $u_{w, D_w(j)}^{\GL, -}$ are all $p$-adically integral and $e_w^-$ defines a projector on $\pi_w^{I_{w,r}}$.

One readily sees that $\pi$ is $P$-anti-ordinary (at $p$) over level $r$ if $\mu_p$ is unramified and each $\pi_w$ is \emph{$P_w$-anti-ordinary} of level $r$, in the sense that $e_w^- \pi_w^{I_{w,r}} \neq 0$.

\begin{lemma} \label{ord anti ord duality}
    Let $\pi_w$ as above. Then, the representation $\pi_w$ is $P_w$-anti-ordinary of some level $r \geq 0$ if and only if its contragredient $\pi_w^\vee$ is $P_w$-ordinary of level $r$. In that case, $\pi_w$ is $P$-anti-ordinary of all level $r \gg 0$.
\end{lemma}

\begin{proof}
    This is a simple generalization of \cite[Lemma 8.3.6 (i)]{EHLS}. The proof goes through verbatim by replacing the pro-$p$ Iwahori subgroup (also denoted $I_{w,r}$) by $I_{P_w, w, r}$ and only considering the Hecke operators $u_{w, D_w(j)}^{\GL, -}$ and $u_{w, D_w(j)}^{\GL}$, for $1 \leq j \leq r_w$. The key part is that all these operators commute with one another.
\end{proof}

\subsubsection{Conventions on contragredient pairings.} \label{conv on contra pairings}
In what follows, given any representation $\rho$, we denote its contragredient representation by $\rho^\vee$. For instance, let $\sigma_w$ be an admissible irreducible supercuspidal representation of $L_w(\KK_w)$ and $\sigma_w^\vee$ be its contragredient, also an admissible irreducible supercuspidal representation of $L_w(\KK_w)$.

Let $\brkt{\cdot}{\cdot}_{\sigma_w} : \sigma_w \times \sigma_w^\vee \to \CC$ be the tautological pairing on a pair of contragredient representations. Define
\begin{align*}
    \brkt{\cdot}{\cdot}_{w} &: 
        \ind{P_w}{G_w} \sigma_w 
            \times 
        \ind{P_w}{G_w} \sigma_w^\vee
    \to
        \CC \\
    \brkt{\varphi}{\varphi^\vee}_{w} &=
        \int_{G_w(\OO_w)}
            \brkt{\varphi(k)}{\varphi^\vee(k)}_{\sigma_w}
        dk \ ,
\end{align*}
a perfect $G_w(\KK_w)$-equivariant pairing. Here $dk$ is the Haar measure on $G_w(\OO_w)$ that such that $\vol(G_w(\OO_w)) = 1$ with respect to $dk$. Then $\brkt{\cdot}{\cdot}_{w}$ naturally identifies $\ind{P_w}{G_w} \sigma_w^\vee$ as the contragredient of $\ind{P_w}{G_w} \sigma_w$.

Let $\pi_w$ be a the constituent at $w \in \Sigma_p$ of $\pi_p$ as above. From now on, we assume $\pi_w$ is the unique irreducible quotient $\ind{P_w}{G_w} \sigma_w \twoheadrightarrow \pi_w$. Equivalently, $\pi_w^\vee$ is the unique irreducible subrepresentation $\pi_w^\vee \hookrightarrow \ind{P_w}{G_w} \sigma_w^\vee$, see Remark \ref{sc support assumption}. If one restricts the second argument of $\brkt{\cdot}{\cdot}_{w}$ to $\pi_w^\vee$, then the first argument factors through $\pi_w$. In other words, $\brkt{\cdot}{\cdot}_{w}$ induces the tautological pairing $\brkt{\cdot}{\cdot}_{\pi_w} : \pi_w \times \pi_w^\vee \to \CC$ and
\[
    \brkt{\phi}{\phi^\vee}_{\pi_w} =
        \int_{G_w(\OO_w)}
            \brkt{\varphi(k)}{\varphi^\vee(k)}_{\sigma_w}
        dk \ , \ \ \ \forall \phi \in \pi_w, \phi^\vee \in \pi_w^\vee \ ,
\]
where $\varphi$ is any lift of $\phi$ and $\varphi^\vee$ is the image of $\phi^\vee$.

Let $(\tau_w, X_w)$ be the BK-type of $\sigma_w$, a representation of $L_w(\OO_w)$. Then, its contragredient $(\tau_w^\vee, X_w^\vee)$ is the BK-type of $\sigma_w^\vee$. One can find $L_w(\OO_w)$-embeddings $\tau_w \hookrightarrow \sigma_w$ and $\tau_w^\vee \hookrightarrow \sigma_w^\vee$ (both unique up to scalar) such that for all $\alpha \in X_w$, $\alpha^\vee \in X_w^\vee$,
\[
    \brkt{\alpha}{\alpha^\vee}_{\sigma_w}
        =
    \brkt{\alpha}{\alpha^\vee}_{\tau_w} \ .
\]

More generally, upon restriction of $\sigma_w$ and $\sigma_w^\vee$ to representations of $L_w(\OO_w)$, there are a direct sum decomposition 
\[
    \sigma_w = \bigoplus_{\uptau_w} \sigma_w[\uptau_w]
    \ \ \ \text{and} \ \ \
    \sigma_w^\vee = \bigoplus_{\uptau_w} \sigma_w^\vee[\uptau_w]
\]
where $\uptau_w$ runs over all smooth irreducible representations of $L_w(\OO_w)$ and the square brackets $[\cdot]$ denote isotypic subspaces. The restriction of $\brkt{\cdot}{\cdot}_{\sigma_w}$ to 
$
    \sigma_w[\uptau_w]
        \times
    \sigma_w^\vee[\uptau'_w]
$
is identically zero if $\uptau'_w \not\cong \uptau_w^\vee$. On the other hand, its restriction to
$
    \sigma_w[\uptau_w]
        \times
    \sigma_w^\vee[\uptau^\vee_w]
$
is a perfect $L_w(\OO_w)$-invariant pairings.

\subsubsection{Structure theorem for $P$-anti-ordinary representations.}
Since $\pi_w^{I_{w,r}}$ is stable under the action of $I_{w,r}^0/I_{w,r} \cong L_w(\OO_w/\p_w^r\OO_w)$, it decomposes as a direct sum of isotypic subspaces over all irreducible representations of $L_w(\OO_w/\p_w^r\OO_w)$. Given such a representation $\tau_w$, we say that $\phi \in \pi_w^{I_{w,r}}$ is $P_w$-anti-ordinary \emph{of type} $\tau_w$ if it is $P_w$-anti-ordinary and it lies in the isotypic subspace $\pi_w^{I_{w,r}}[\tau_w]$.

\begin{theorem} \label{Pw a ord vector for G1}
    Let $w \in \Sigma_p$ and $\pi_w$ be a constituent of $\pi$ as above. Assume that the weight $\kappa$ of $\pi$ satisfies Inequality \eqref{ineq kappa sigma}. Assume that $\pi_w$ is $P_w$-anti-ordinary of level $r \gg 0$ and is the unique irreducible quotient $\ind{P_w}{G_w} \sigma_w \twoheadrightarrow \pi_w$ as above. Assume that the BK-type $(\tau_w, X_w)$ of $\pi_w$ factors through $L_w(\OO_w/\p_w^r\OO_w)$. Given any $\alpha \in X_w$, let $\varphi_{w,r}^{\Pwaord} \in \ind{P_w}{G_w} \sigma_w$ be the unique vector with support $P_wI_{w,r}$ such that $\varphi_{w,r}^{\Pwaord}(1) = \alpha$ and $\varphi_{w,r}^{\Pwaord}$ is fixed by $I_{w,r}$.
    
    The image $\phi_{w,r}^{\Pwaord} \in \pi_w^{I_{w,r}}$ of $\varphi_{w,r}^{\Pwaord}$ is $P_w$-anti-ordinary of level $r$ for any $r$ such that $\tau_w$ factors through $L_w(\OO_w/\p_w^r\OO_w)$. It satisfies :
    \begin{enumerate}[label=(\roman*)]
        \item Let $\phi^\vee \in \pi_{w}^{\vee, I_{w,r}}$ and denote its image in $\ind{P_w}{G_w} \sigma_w$ by $\varphi^\vee$. Then, 
        \[
            \brkt{
                \phi_{w,r}^{\Pwaord}
            }{
                \phi^\vee
            }_{\pi_w} 
            = 
                \vol(I_{w,r}^0)
                \brkt{\alpha}{\varphi^\vee(1)}_{\sigma_w}
        \] 

        In particular, $\brkt{\phi_{w,r}^{\Pwaord}}{\phi}_{\pi_w} \neq 0$ if and ony if $\phi^\vee$ is $P_w$-ordinary and the component of $\varphi^\vee(1)$ in $\sigma_w^\vee[\tau_w^\vee]$ is non-zero.
        
        \item The vector $\phi_{w,r}^{\Pwaord}$ lies in the $\tau_w$-isotypic space of $\pi_w^{I_{w,r}}$. Moreover, any other $P_w$-anti-ordinary vector of type $\tau_w$ is obtained as above for some other choice of $\alpha' \in X_w$.
        
        \item One can pick different choices of $\alpha$ for each $r' \geq r$ so that
        \[
            \sum_{
                \gamma \in I_{w,r}/(I_{w,r'}^0 \cap I_{w,r})
            }
                \pi_w(\gamma) \phi_{w,r'}^{\Pwaord} = 
                \phi_{w,r}^{\Pwaord}
        \]
    \end{enumerate}
\end{theorem}

\begin{proof}
    Write $\phi_{w,r}$ and $\varphi_{w,r}$ instead of $\phi_{w,r}^{\Pwaord}$ and $\varphi_{w,r}^{\Pwaord}$ respectively. We first show that property (i) holds. By Lemma \ref{ord anti ord duality}, $\pi_w^\vee$ is $P_w$-ordinary of level $r$. Write
    \[
        \pi_w^{\vee, I_{w,r}} = \bigoplus_{a=1}^A V_a \ ,
    \]
    where each $V_a$ is a simultaneous generalized eigenspace for the Hecke operators $u_{w,D_w(j)}^{\GL}$. 
    
    From the proof of Theorem \ref{holo canonical vector} and the remark that follow, exactly one $V_a$ has generalized eigenvalues that are all $p$-adic units. We may assume that this holds true for $V_1$. The exact eigenvalue of $u_{w, D_w(j)}^{\GL}$ is given by Equation \eqref{central character condition}, denote it $\beta_{w, D_w(j)}$. For $1 < a \leq A$, at least one generalized eigenvalue for $V_a$ is not a $p$-adic unit.
    
    Given $\phi^\vee \in \pi_w^{\vee, I_{w,r}}$, write it as a sum
    \[
        \phi^\vee = 
            \sum_{a=1}^A \phi_a^\vee \ , 
    \]
    with $\phi_a^\vee \in V_a$. Let $\varphi_a^\vee$ denote the images of $\phi_a^\vee$ in $\ind{P_w}{G_w} \sigma_w^\vee$. Then,
    \[
        \brkt{\phi_{w,r}}{\phi^\vee}_{\pi_w} 
        = 
            \sum_{a=1}^A
                \brkt{
                    \phi_{w,r}
                }{
                    \phi_a^\vee
                }_{\pi_w} 
        = 
            \sum_{a=1}^A
            \int_{G_w(\OO_w)}
                \brkt{
                    \varphi_{w,r}(k)
                }{
                    \varphi_a^\vee(k)
                }_{\sigma_w}
            dk
    \]
    
    Recall that the support of $\varphi_{w,r}$ is $P_wI_{w,r}$. Also, the intersection of $P_wI_{w,r}$ with $G_w(\OO_w)$ is equal to $I_{w,r}^0$ and by Theorem \ref{holo canonical vector} (ii), $\varphi_a^\vee(I_{w,r}^0) = 0$ for all $a \neq 1$. Therefore,
    \[
        \brkt{\phi_{w,r}}{\phi^\vee}_{\pi_w} 
        =
            \int_{I_{w,r}^0}
                \brkt{
                    \varphi_{w,r}(k)
                }{
                    \varphi_1^\vee(k)
                }_{\sigma_w}
            dk
    \]

    Since $I_{w,r}^0 = L_w(\OO_w) I_{w,r}$ and $\varphi_{w,r}^{\Pwaord}$, $\varphi_1^\vee$ are both fixed by $I_{w,r}$, one obtains
    \begin{align*}
        \brkt{\phi_{w,r}}{\phi^\vee}_{\pi_w} 
        =
            \int_{I_{w,r}^0}
                \brkt{
                    \varphi_{w,r}(1)
                }{
                    \varphi_1^\vee(1)
                }_{\sigma_w}
            dk
        =
            \vol(I_{w,r}^0)
                \brkt{
                    \alpha
                }{
                    \varphi_1^\vee(1)
                }_{\sigma_w} \ .
    \end{align*}

    The desired relation holds by noting that $\varphi_1^\vee(1) = \varphi^\vee(1)$. The second part of (i) follows immediately from the discussion about isotypic subspaces at the end of Section \ref{conv on contra pairings}.

    As a consequence of property (i), we immediately obtain $\brkt{\phi_{w,r}}{V_a}_{\pi_w} = 0$ for all $a > 1$. Furthermore, for all $\phi^\vee \in V_1$, we have
    \[
        \brkt{
            u_{w, D_w(j)}^{\GL, -} 
            \phi_{w,r}
        }{
            \phi^\vee
        }_{\pi_w}
            =
        \brkt{
            \phi_{w,r}
        }{
            u_{w, D_w(j)}^{\GL} 
            \phi^\vee
        }_{\pi_w}
            =
        \beta_{w, D_w(j)}
        \brkt{
            \phi_{w,r}
        }{
            \phi^\vee
        }_{\pi_w} \ .
    \]

    By combining these two facts, we obtain
    \[
        \brkt{
            u_{w, D_w(j)}^{\GL, -} 
            \phi_{w,r}
        }{
            \phi^\vee
        }_{\pi_w}
            =
        \beta_{w, D_w(j)}
        \brkt{
            \phi_{w,r}
        }{
            \phi^\vee
        }_{\pi_w} \ .
    \]
    for all $\phi^\vee$ in $\pi_w^{\vee, I_{w,r}}$. In other words, $\phi_{w,r}$ is $P_w$-anti-ordinary.

    Furthermore, note that the argument above implies that the subspace of $P_w$-anti-ordinary vectors of type $\tau_w$ in $\pi_w^{I_{w,r}}$ is dual to the subspace of $P_w$-ordinary vectors of type $\tau_w^\vee$. From Theorem \ref{holo canonical vector}, they both have dimension $\dim \tau_w = \dim \tau_w^\vee$. Since the space generated by the action of $L_w(\OO_w/\p_w^r\OO_w)$ on $\phi_{w,r}$ is of dimension $\dim \tau_w$ and consists of $P_w$-anti-ordinary vectors of type $\tau_w$. Given $l \in L_w(\OO_w/\p_w^r\OO_w)$, one readily sees that $\pi_w(l) \phi_{w,r}$ is the $P_w$-anti-ordinary vector obtained by picking $\alpha' = \tau_w(l)\alpha$ in $X_w$ instead of $\alpha$. This proves the second sentence of part (ii).

    Finally, part (iii) and the first statement of part (ii) follow immediately from the fact that the analogous properties hold for $\varphi_{w,r}$. 
\end{proof}

Keeping the assumption and notation of Lemma \ref{Pw a ord vector for G1}, fix a vector $\alpha \in X$. Using Lemma \ref{ord anti ord duality}, $\pi_w^\vee \hookrightarrow \ind{P_w}{G_w} \sigma_w^\vee$ is $P_w$-ordinary. Let $(\tau_w^\vee, X^\vee)$ be the BK-type of $\pi_w^\vee$ and fix any $\alpha^\vee \in X^\vee$ such that $\brkt{\alpha}{\alpha^\vee}_{\tau_w} = 1$. Let $\phi_{w,r}^{\vee, \Pword}$ be the $P_w$-ordinary vector associated to $\alpha^\vee$ obtained from Theorem \ref{holo canonical vector} (iii). 

In fact, as $r$ increases, one may pick compatible choices of $\alpha$ so that property (iii) of Theorem \ref{Pw a ord vector for G1} holds and compatible choices of $\alpha^\vee$ such that $\brkt{\alpha}{\alpha^\vee}_{\tau_w} = 1$ for all $r \gg 0$. Then, as a consequence of Theorem \ref{Pw a ord vector for G1} (i),
\[
    \frac{
        \brkt
        {
            \phi_{w,r}^{\Pwaord}
        }{
            \phi_{w,r}^{\vee, \Pword}
        }_w 
    }{
        \vol(I_{w,r}^0)
    }
    = 
        \brkt{\alpha}{\alpha^\vee}_{\sigma_w}
    =
        \brkt{\alpha}{\alpha^\vee}_{\tau_w} = 1
\]  
is independent of $r \gg 0$.

Furthermore, one readily obtains a result analogous to Theorem \ref{Pord structure thm} from Theorem \ref{Pw a ord vector for G1}. Let $\tau = \bigotimes_{w \in \Sigma_p} \tau_w$ be the BK-type of $\pi_p$, using the identification \eqref{facto G(Qp)}, as explained ahead of Theorem \ref{Pord structure thm}.

\begin{corollary} \label{Paord structure thm}
    Let $\pi$ be an anti-holomorphic cuspidal representation of $G$ of weight type $(\kappa, K_r)$ for some $r \gg 0$. Suppose $\kappa$ satisfies Inequality \eqref{ineq kappa sigma}. Then, $\pi$ is $P$-anti-ordinary if and only if $\pi^\flat$ is $P$-ordinary. Let $\tau = \bigotimes_{w \in \Sigma_p} \tau_w$ be the BK-type of $\pi$. 
    
    There exists a unique (up to the action of $L_P(\ZZ_p)$) $P$-anti-ordinary vector $\phi_{r}^{\Paord}$ of level $r$ and type $\tau$ in $\pi_{p}^{I_{P, r}}$. Furthermore, there exists $P_w$-ordinary vectors $\phi_{w, r}^{\Pwaord}$ of level $r$ and type $\tau_w$ as in Theorem \ref{Pw a ord vector for G1} such that, under the identification $\pi_p = \mu_p \otimes \bigotimes_{w \in \Sigma_p} \pi_w$, $\phi_{r}^{\Paord} = \bigotimes_{w \in \Sigma_p} \phi_{w, r}^{\Pwaord}$.
\end{corollary}

\subsection{$P$-(anti-)ordinary representations on $G_2$.} \label{Pord and Paord on G_2}
In this section we compare the theory of $P$-(anti-)ordinary representations on $G_1$ and $G_2$, where $G_i$ is the unitary group associated to the PEL datum $\PP_i$ introduced in Section \ref{P1 P2 P3 P4}. We add a subscript $V$ (resp. $-V$) in our notation whenever we want to emphasize that we are working with $G_1$ (resp. $G_2$).

\subsubsection{Comparison between representations of $G_1$ and $G_2$.} \label{P ord and anti ord on G2}
Note that there is a canonical identification $G_1(\AA) = G_2(\AA)$. Furthermore, the identification from isomorphism \eqref{prod G over Zp} remains the same for both $G_1$ and $G_2$. However, the opposite choices of $\OO_\KK \otimes \ZZ_p$-lattices $L_1^\pm = L_2^\mp$ introduce many changes in the notation. 

For instance, under the identification $G_1(\AA) = G_2(\AA)$, the group $H_{0, -V} = H_0$ for $G_1$ corresponds to $H_{0, -V}$ (by switching the roles of $\Lambda_0$ and $\Lambda_0^\vee$.) However, the identification from isomorphism \eqref{H0 over S0} interchanges the role of $\sigma \in \Sigma_\KK$ and $\sigma c$ (where $c$ denotes complex conjugation). 

Given a dominant weight $\kappa$ of $T_1 := T_{H_0, V}$, it is identified with a tuple $(\kappa_0, (\kappa_\sigma)_\sigma)$ where $\kappa_0 \in \ZZ$ and $\kappa_{\sigma} \in \ZZ^{b_\sigma}$. The torus $T_2 := T_{H_0, -V}$ is equal to $T_1$ but the corresponding isomorphism \eqref{H0 over S0} for $G_2$ identifies $\kappa$ with $(\kappa_0, (\kappa_{\sigma c})_\sigma)$. We denote the latter by $\kappa^\flat$. In particular, $\kappa_{\sigma c} \in \ZZ^{a_\sigma} = \ZZ^{b_{\sigma c}}$ and $\kappa^\flat$ is dominant with respect to $B_{H_0, -V}^\opp$.

As explained in \cite[Sections 6.2.1-6.2.2]{EHLS}, if $\pi$ is a cuspidal (anti-)holomorphic automorphic representation for $G_1$ of weight $\kappa$, then $\pi^\flat = \pi^\vee \otimes ||\nu||^{a(\kappa)}$ (as in Section \ref{comp cohomology theories}) is naturally a cuspidal (anti-)holomorphic automorphic representation for $G_2$ of weight $\kappa$.

Furthermore, by choosing the same partitions $\d_w$ introduced in Section \ref{level at p}, the parabolic subgroup $P_w \subset \GL_n(\OO_w)$ for $G_1$ corresponding to $w \in \Sigma_p$ is replaced by the opposite parabolic subgroups, which in our case is simply its transpose $\tp{P}_w \subset \GL_n(\OO_w)$, when working with $G_2$. Similarly, $P$ is replaced by $\tp{P}$ and the (pro-$p$) $P$-Iwahori subgroup of level $r$ is replaced by the (pro-$p$) $\tp{P}$-Iwahori subgroup of level $r$.

In particular, if $\pi_p \cong \mu_p \otimes \bigotimes_{w \in \Sigma_p} \pi_w$ is the identification obtained from \eqref{facto G(Qp)} for $G_1$, the corresponding factorization on $G_2$ induces
\[
    \pi_p^\flat \cong \mu_p^\flat \otimes \bigotimes_{w \in \Sigma_p} \pi_w^\flat \ ,
\]
where $\pi_w^\flat = \pi_w^\vee$ and $\mu_p^\flat = \mu_p^{-1}|\nu|_p^{a(\kappa)}$.

\subsubsection{Holomorphic and $\tp{P}$-ordinary representations for $G_2$.} 
We keep the notation of Section \ref{P ord and anti ord on G2}. The discussion above shows that $\pi_w$ is $P_w$-ordinary of level $r \gg 0$ if and only if $\pi_w^\flat$ is $\tp{P}_w$-ordinary of level $r \gg 0$. Note that, adapting our definitions in Section \ref{Pord rep defn} from $G_1$ to $G_2$, the latter notion requires to change $P_w$ for $\tp{P}_w$ and the double coset operators $U_{w,j}^{\GL}$ for $U_{w,j}^{\flat, \GL} = \tp{I}_{w,r} t_{w,j}^{-1} \tp{I}_{w,r}$.

We assume that $\pi_w$ is $P_w$-ordinary of level $r \gg 0$, that $\pi_w$ is the unique irreducible subrepresentation of $\ind{P_w}{G_w} \sigma_w$ for some admissible irreducible supercuspidal representation $\sigma_w$, and $\kappa$ satisfies Inequality \eqref{ineq kappa sigma}. The analogue of Theorem \ref{holo canonical vector} is the following.

\begin{lemma} \label{P ord for G2}
    Using the notation above, let $(\tau_w, X_w)$ be the BK-type of $\pi_w$.

    \begin{enumerate}
        \item[(i)] The unique irreducible quotient of $\ind{P_w}{G_w} \sigma_w^\vee$ is isomorphic to $\pi_w^\flat$.
        
        \item[(ii)] Let $(\tau_w^\vee, X_w^\vee)$ be the contragredient of $(\tau_w, X_w)$, the BK-type of $\sigma_w^\vee$. Consider $X_w^\vee$ as a subspace of the vector space associated to $\sigma_w^\vee$, via a fix embedding (unique up to scalar) $\tau_w^\vee \hookrightarrow \sigma_w^\vee$. 
        
        For any $\alpha^\vee \in X_w^\vee$, let $\varphi^\flat_w \in \ind{P_w}{G_w} \sigma_w^\vee$ be the unique function with support $P_w\tp{I}_{w,r}$ (for all $r \gg 0$) such that $\varphi^\flat_{w}(1) = \alpha^\vee$ and $\varphi^\flat_{w}$ is fixed by $\tp{I}_{w,r}$ (for all $r \gg 0$). Let $\phi^\flat_w$ denote its image in $\pi_w^\flat$.

        Then, $\phi^\flat_w$ is $\tp{P}_w$-ordinary of type $\tau_w^\vee$ of level $r \gg 0$. This induces a natural isomorphism between $\tau_w^\vee$ and the subspace of $\tp{P}_w$-ordinary vectors of type $\tau_w^\vee$ of level $r \gg 0$. In particular, the latter is independent of $r \gg 0$ and has dimension $\dim \tau_w^\vee = \dim \tau_w$.
    \end{enumerate}
\end{lemma}

\begin{proof}
    Consider the composition of $\pi_w \hookrightarrow \ind{P_w}{G_w} \sigma_w$ with the map (of vector spaces)
    \begin{align*}
        \ind{P_w}{G_w} \sigma_w 
            &\to 
                \ind{\tp{P}_w}{G_w} \sigma^\vee_w \\
        \phi 
            &\mapsto
                \phi^\vee(g) := \phi(\tp{g}^{-1})
    \end{align*}

    Its image is $\pi_w^\flat = \pi_w^\vee$ and realizes $\pi_w^\flat$ as the unique irreducible subrepresentation of $\ind{\tp{P}_w}{G_w} \sigma^\vee_w$. In particular, all the consequences of Theorem \ref{holo canonical vector} hold for $\pi_w^\flat$ by replacing $P_w$ by $\tp{P}_w$ and $\sigma_w^\vee$ by $\sigma_w^\vee$. 
    
    Given $\alpha^\vee \in X_w^\vee$ as above, let $\phi_w^\vee \in \pi_w^{\flat, I_{w,r}}$ and $\varphi_w^\vee \in \ind{\tp{P}_w}{G_w} \sigma^\vee_w$ be the vectors obtained from Theorem \ref{holo canonical vector} (iii) associated to $\alpha^\vee$. In particular, $\phi_w^\vee$ is a $\tp{P}_w$-ordinary vector of type $\tau_w^\vee$ and the subspace generated by the action of $L_w(\OO_w)$ on $\phi_w^\vee$ is exactly of all $\tp{P}_w$-ordinary vectors of type $\tau_w^\vee$. In particular, the latter is independent of $r \gg 0$ and isomorphic to $\tau_w^\vee$ over $L_w(\OO_w)$.

    Now, consider the standard intertwining operator $\ind{P_w}{G_w} \sigma^\vee_w \xrightarrow{\sim} \ind{\tp{P_w}}{G_w} \sigma^\vee_w$. It identifies $\pi_w^\flat$ as the unique irreducible quotient of $\ind{\tp{P}_w}{G_w} \sigma^\vee_w$. Furthermore, the vector $\varphi_w^\vee \in \ind{\tp{P_w}}{G_w} \sigma^\vee_w$ exactly corresponds to the vector $\varphi_w^\flat \in \ind{P_w}{G_w} \sigma^\vee_w$ described above. Then, $\phi_w^\flat = \phi_w^\vee$ is the desired vector and this concludes the proof.
\end{proof}

\subsubsection{Anti-holomorphic and $\tp{P}$-anti-ordinary representations for $G_2$.} 
Going back to the discussion of Section \ref{P ord and anti ord on G2}, we know that $\pi_w$ is $P_w$-anti-ordinary of level $r \gg 0$ (for $G_1$) if and only if $\pi_w^\flat$ is $\tp{P}_w$-anti-ordinary of level $r \gg 0$ (for $G_2$). Again, adapting our definitions in Section \ref{Pord rep defn} from $G_1$ to $G_2$, the latter notion requires to change $P_w$ for $\tp{P}_w$ and the double coset operators $U_{w,j}^{\GL, -}$ for $U_{w,j}^{\flat, \GL, -} = \tp{I}_{w,r} t_{w,j} \tp{I}_{w,r}$.

We assume that $\pi_w$ is $P_w$-anti-ordinary of level $r \gg 0$, that $\pi_w$ is the unique irreducible quotient of $\ind{P_w}{G_w} \sigma_w$ for some admissible irreducible supercuspidal representation $\sigma_w$, and $\kappa$ satisfies Inequality \eqref{ineq kappa sigma}. The analogue of Theorem \ref{Pw a ord vector for G1} is the following.

\begin{lemma} \label{P a ord for G2}
    Using the notation above, let $(\tau_w, X_w)$ be the BK-type of $\pi_w$.

    \begin{enumerate}
        \item[(i)] The unique irreducible subrepresentation of $\ind{P_w}{G_w} \sigma_w^\vee$ is isomorphic to $\pi_w^\flat$.
        
        \item[(ii)] Let $(\tau_w^\vee, X_w^\vee)$ be the contragredient of $(\tau_w, X_w)$, the BK-type of $\sigma_w^\vee$. Consider $X_w^\vee$ as a subspace of the vector space associated to $\sigma_w^\vee$, via a fix embedding (unique up to scalar) $\tau_w^\vee \hookrightarrow \sigma_w^\vee$. 
        
        For each $r \gg 0$ and $\alpha \in X_w^\vee$, there exists some unique $\tp{P}_w$-anti-ordinary $\phi^\flat_{w,r} \in \pi_w^{I_r}$ of type $\tau_w^\vee$ and level $r$ such that $\varphi^\flat_{w,r}(1) = \alpha$, where $\varphi^\flat_{w,r}$ is the image of $\phi^\flat_{w,r}$ in $\ind{P_w}{G_w} \sigma^\vee_w$, and the support of $\phi^\flat_{w,r}$ contains $P_w \tp{I}_{w,r}$.

        \item[(iii)] For $r' > r \gg 0$, one can choose $\alpha$, $\alpha' \in X_w^\vee$ such that the vectors $\phi^\flat_{w,r}$ and $\phi^\flat_{w,r'}$ corresponding to $\alpha$ and $\alpha'$ respectively satisfy
        \[
            \sum_{
                \gamma \in \tp{I}_{w,r}/(\tp{I}_{w,r'}^0 \cap \tp{I}_{w,r})
            }
                \pi^\flat_w(\gamma) \phi_{w,r'}^\flat = 
                \phi_{w,r}^\flat
        \]
    \end{enumerate}
\end{lemma}

\begin{proof}
    As in the proof of Lemma \ref{P ord for G2}, the map
    \begin{align*}
        \ind{\tp{P}_w}{G_w} \sigma^\vee_w 
            &\to 
                \ind{P_w}{G_w} \sigma_w \\
        \phi 
            &\mapsto
                \phi^\vee(g) := \phi(\tp{g}^{-1}) \ ,
    \end{align*}
    realizes $\pi_w^\flat = \pi_w^\vee$ as the unique irreducible quotient of $\ind{\tp{P}_w}{G_w} \sigma^\vee_w$. 
    
    In particular, all the consequences of Theorem \ref{Pw a ord vector for G1} hold for $\pi_w^\flat$ by replacing $P_w$ by $\tp{P}_w$ and $\sigma_w$ by $\sigma_w^\vee$. Given $\alpha \in X_w^\vee$ as above, let $\varphi_{w,r}' \in \ind{\tp{P}_w}{G_w} \sigma^\vee_w$ be the vectors obtained from Theorem \ref{Pw a ord vector for G1} associated to $\alpha$.
    
    Furthermore, consider the standard intertwining operator 
    $
        \ind{\tp{P_w}}{G_w} \sigma^\vee_w
            \xrightarrow{\sim} 
        \ind{P_w}{G_w} \sigma^\vee_w 
    $. Its image is both the unique irreducible quotient of $\ind{\tp{P_w}}{G_w} \sigma^\vee_w$, namely $\pi_w^\flat$, and the unique irreducible subrepresentation of $\ind{P_w}{G_w} \sigma^\vee_w$. This proves part (i).
    
    To conclude, let $\phi_{w,r}^\flat$ (resp. $\varphi_{w,r}^\flat$) be the image of $\varphi_{w,r}'$ in $\pi_w^\vee$ (resp. $\ind{P_w}{G_w} \sigma^\vee_w$) via this intertwining operator. The fact that $\phi_{w,r}^\flat$ is $\tp{P}_w$-anti-ordinary of type $\tau_w^\vee$ and level $r$ follows from Theorem \ref{Pw a ord vector for G1} (ii). Similarly, part (iii) follows from Theorem \ref{Pw a ord vector for G1} (iii) (upon making the appropriate adjustement between $G_1$ and $G_2$). The properties of $\varphi_{w,r}'$ are obtained from an easy computation using the definition of $\varphi_{w,r}'$ and the exact formula for the intertwining operator above.
\end{proof}

\begin{remark}
    In Theorem \ref{Pw a ord vector for G1}, Lemma \ref{P ord for G2} and Lemma \ref{P a ord for G2}, more general statement can be made for any type. However, for applications to computations of zeta integrals in forthcoming work of the author, \cite{Mar23c}, the results above only involving the BK-type of a fixed representation are sufficient.
\end{remark}

\section{Comparison of $P$-(anti-)ordinary modular and automorphic forms.} \label{Pord mod forms}
In this section, we work with $G = G_1$ and we use the same notation as in Section \ref{Pord rep defn} without comments. The material here adapts some of the theory of \cite[Section 6.6]{EHLS} for any parabolic subgroup $P$ as in Section \ref{level at p}. 

In particular, we identify integral spaces of $P$-ordinary cusp forms of level $K_r$ with a fixed weight $\kappa$ and $P$-nebentypus $\tau$ as lattices inside certain holomorphic $P$-ordinary cuspidal automorphic representations $\pi$ of weight type $(\kappa, K_r)$ whose BK-type is $\tau$. Using characters of Hecke algebra associated to $\pi$ one can study congruences between such $P$-ordinary cusp forms modulo $p$.

\subsection{Hecke algebras.}
Let $\kappa$ be a dominant character of $T_{H_0}$ and $\tau$ be an irreducible smooth representation of $L_P(\ZZ_p)$. Fix $r \gg 0$ such that $\tau$ factors through $L_P(\ZZ_p/p^r\ZZ_p)$ and let $K = K_r = K^pI_r \subset G(\AA_f)$ be a neat compact open level subgroup. Let $R \subset \CC$ be an $S^0[\tau]$-algebra.

\subsubsection{Hecke algebras on cusp forms.} \label{hecke alg on cusp forms}
As in \cite[Section 2.6.8]{EHLS}, for all $g \in G(\AA_f^p)$, the double coset operator $T_r(g) = [K_rgK_r]$ naturally acts as an endomorphism of $M_\kappa(K_r; R)$. The subspace of cusp forms and the subspace of $P$-nebentypus $\tau$ are both stable under the action of $T_r$. The material of \cite{{EHLS}} only considers the case where $P$ is a Borel subgroup but the same arguments and formulas remain valid in our case using our moduli interpretation of $\EE_{r, \tau}$ from Section \ref{P nebentypus}. This is because $T_r(g)$ only acts on the PEL datum of a given point and not on its $p$-level structure. 

Furthermore, assume that $R$ is a $p$-adic domain. In that case, the arguments of Hida \cite[8.3.1]{Hid04} show that the Hecke operator $u_{w, D_w(j)} = u_{w, D_w(j), \kappa}$ also acts as an endomorphism of $M_\kappa(K_r; R)$, see also \cite[Sections 2.6.9, 2.9.5]{EHLS}. Again, the action of $u_{w, D_w(j)}$ stabilizes the subspace of cusp forms and the subspace of forms with $P$-nebentypus $\tau$.

We now construct the Hecke algebra (of level $K_r$) generated by all Hecke operators at unramified places and at $p$. More precisely, let $l \neq p$ be any prime of $\QQ$ and consider the set $\PP_l$ of all primes of $\KK^+$ above $l$. Write $\PP_l = \PP_{l,1} \coprod \PP_{l,2}$, where $\PP_{l,1}$ is the subset of such primes that split in $\KK$ and $\PP_{l,2}$ is the complement. Therefore, one naturally has an identification
\[
    G(\QQ_l) = \prod_{v \in \PP_{l,1}} \GL_n(\KK_v^+) \times G_{l,2} \ ,
\]
where $G_{l,2}$ is the subgroup of elements $((x_w), t) \in \prod_{w \in \PP_2} \GL_n(\KK_w) \times \QQ_l^\times$ such that each $x_w$ preserve the Hermitian form on $V \times_{\KK} \times \KK_w$ with the same similitude factor $t$. In particular $K_l \subset G(\QQ_l)$ is a product of local factors over all places in $\PP_l$. Let $S_l = S_l(K^l)$ be the subset of $\PP_l$ consisting of all places for which the local factor of $K_l$ is not the maximal hyperspecial subgroup. Let $S_{l, i} = S_l \cap PP_{l, i}$ and define
\[
    G(\QQ_l)^{S_l} = 
    \begin{cases}
        \prod_{
            v \in 
                \PP_{l, 1} \backslash S_{l,1}
        } 
            \GL_n(\KK_v^+) \times G_{l, 2} \ , 
        & 
            \text{if } S_{l, 2} = \emptyset \\
        \prod_{
            v \in 
                \PP_{l, 1} \backslash S_{l,1}
        } 
            \GL_n(\KK_v^+) \ , 
        & 
            \text{otherwise.}
    \end{cases}
\]

Finally, let $S = S(K^p) = \bigcup_{l \neq p} S_l(K^l)$ and define 
\[
    G(\AA_f^S)  = \prod_{l \neq p} G(\QQ_l)^{S_l}\ .
\]

Let $\bTT_{K_r, \kappa, R}$ be the $R$-subalgebra of $\End_\CC(S_\kappa(K_r; \CC))$ generated by the operators $T(g) = T_r(g)$ for all $g \in G(\AA_f)^S$ and $u_{w, D_w(j)}$ for all $w \in \Sigma_p$, $1 \leq j \leq r_w$. Similarly, one defines $\bTT_{K_r, \kappa, \tau, R}$ as the quotient algebra obtained by restricting each operator to an endomorphism of $S_\kappa(K_r, \tau; \CC)$.


\subsubsection{Serre duality and Hecke algebras on anti-ordinary cusp forms.}
Going back to $G = G_1$, the space of \emph{anti-holomorphic} cuspidal forms of weight $\kappa$ and level $K_r$ is defined as
\[
    H^d_{\kappa}(K_r; \CC) := H_!^d(\level{K_r}{\Sh}, \w_{\kappa, r})
\]
and its subspace of $P$-nebentypus $\tau$ is
\[
    H^d_{\kappa}(K_r, \tau; \CC) := H_!^d(\level{K_r}{\Sh}, \w_{\kappa, r, \tau}) \ .
\]

One can define an $R$-integral structure on these spaces by considering the integral models of $\level{K}{\Sh}$. However, we instead follow \cite[Section 6.4.2]{EHLS} and define its integral structure via duality from a normalized Serre duality pairing.

By definition of $\kappa^D$, one can construct a canonical perfect pairing
\[
    \brkt{\cdot}{\cdot}_{\kappa, K_r}^{\text{Ser}} :=
        H^0_!(\level{K_r}{\Sh}(V), \w_{\kappa, r}) 
            \otimes 
        H^d_!(\level{K_r}{\Sh}(V), \w_{\kappa^D, r})
    \to
        \CC
\]

Let $\vol(I_r^0)$ be the volume of $K_r^0 = K^pI_r^0$ with respect to the Tamagawa measure $dg$ from \cite[Section 6.3]{EHLS}. We define $\brkt{\cdot}{\cdot}_{\kappa, K_r}$ as the \emph{normalized Serre pairing}
\begin{align*}
        H^0_!(\level{K_r}{\Sh}(V), \w_{\kappa, r}) 
            &\otimes 
        H^d_!(\level{K_r}{\Sh}(V), \w_{\kappa^D, r})
    \to
        \CC \\
    \brkt{\cdot}{\cdot}_{\kappa, K_r} &:=
        \frac{1}{\vol(I_r^0)} 
        \brkt{\cdot}{\cdot}_{\kappa, K_r}^{\text{Ser}}
\end{align*}

This identifies $H^d_{\kappa}(K_r; \CC)$ as the dual of $S_{\kappa}(K_r; \CC)$, and via this identification we define
\[
    H^d_{\kappa}(K_r; R) := \hom_R(S_{\kappa}(K_r; R), R)
\]

Similarly, $H^d_{\kappa}(K_r, \tau; R)$ is defined by replacing $S_{\kappa}(K_r; R)$ with $S_{\kappa}(K_r, \tau^\vee; R)$. Then, one defines the $R$-Hecke algebra $\bTT_{K_r, \kappa, R}^d$ by proceeding as in the definition of $\bTT_{K_r, \kappa, R}^d$ but replacing $S_\kappa(K_r; R)$ with $H^d_{\kappa}(K_r; R)$ and $u_{w, D_w(j)}$ by $u^-_{w, D_w(j)}$. Upon restriction to $H^d_{\kappa}(K_r; R)$, one obtains the quotient algebra $\bTT_{K_r, \kappa, \tau, R}^d$.

\begin{lemma} \label{iso holo hecke alg and aholo hecke alg}
    Let $R \subset \CC$ be an $S_0$-algebra as above. There exists a unique $R$-algebra isomorphism $\bTT_{K_r, \kappa, R} \xrightarrow{\sim} \bTT^d_{K_r, \kappa^D, R}$ such that $u_{w, D_w(j)}$ is mapped to $u_{w, D_w(j)}^-$ and $T(g)$ to $||\nu(g)||^{a(\kappa)} \cdot T(g^{-1})$. If $R$ is an $S_0[\tau]$-algebra, it induces an isomorphism of $R$-algebra $\bTT_{K_r, \kappa, \tau, R} \xrightarrow{\sim} \bTT^d_{K_r, \kappa^D, \tau, R}$.
\end{lemma}

\begin{proof}
    The proof is exactly the same as the one of Lemma 6.6.1 (i) in \cite{EHLS}. It is an immediate consequence of Serre duality.
\end{proof}

\subsubsection{Automorphic representations as Hecke modules}
In what follows, for all the Hecke algebras $\bTT_{\bullet}^?$, let $\bTT_{\bullet}^{?, p}$ denote the $R$-subalgebra generated only by the operators $T(g)$ for $g \in G(\AA_f^S)$. Moreover, we omit the subscript $R$ when it $R = S^0$ (or $S^0[\tau]$). We also use the notation from Section \ref{comp cohomology theories} without comments.

Let $\pi$ be a holomorphic cuspidal automorphic representation of $G$ of weight type $(\kappa, K_r)$. Recall that it is defined over some number field $E(\pi)$ containing $\KK'$, see Remark \ref{E pi rational}. Recall the definition of $S = S(K^p)$ above and consider the factorization
\[
    \pi = \pi_\infty \otimes \pi_f \ \ \ ; \ \ \ \pi_f = \pi_p \otimes \pi_S \otimes \pi_f^S \ ,
\]

By definition, $K^S$ is the factor of $K^p$ over all places of $\KK^+$ where $K^p$ contains an hyperspecial maximal subgroup. In particular, $(\pi_f^S)^{K^S}$ is a 1-dimensional space spanned by an $E(\pi)$-rational spherical vector. The natural action of $T(g)$ for all $g \in G(\AA_f^S)$ on $\pi_f^{K_r}$ is through its action as a character on $\pi_f^{K^S}$. In other words, this defines a character $\lambda_\pi^p$ of $\bTT_{K_r, \kappa}^p$. Although this definition technically depends on $r$ these homomorphisms are compatible as $r$ increases in the obvious sense, hence we do not include it in the notation of $\lambda_\pi^p$. 

Now, fix a choice of $E(\pi)$-rational spherical vector in $\pi_f^{S}$ and a choice of basis for the $1$-dimensional $H^0(\P_h, K_h; \pi_\infty \otimes W_\kappa)$. Let $S_\kappa(K_r, \CC)(\pi)$ be the $\lambda_\pi^p$-isotypic subspace of $S_\kappa(K_r, \CC)$ as a $\bTT_{K_r, \kappa}^p$-module. Then, the isomorphism \eqref{H! vs H(b, K)} induces an embedding
\[
    j_\pi : 
        \pi_S^{K_S} 
            \otimes 
        \pi_p^{I_r} 
    \hookrightarrow 
        S_\kappa(K_r, \CC)(\pi)
\]
of $\bTT_{K_r, \kappa}^p$-module.

\begin{hypothesis}[Multiplicity one] \label{mult one hyp}
    We say that $\pi$ satisfies the \emph{multiplicity one hypothesis (for $\pi$)} if for any holomorphic cuspidal $\pi' \neq \pi$ of type $(\kappa, K_r)$, the characters $\lambda^p_{\pi'}$ and $\lambda^p_{\pi}$ are distinct.
\end{hypothesis} 

One immediately obtains that if $\pi$ satisfies the multiplicity one hypothesis, then the embedding $j_\pi$ is in fact an isomorphism.

\subsection{$P$-ordinary case.} \label{Hecke alg for Pord}
In this section, we extend the study of the isomorphism $j_\pi$ by also considering the action of the Hecke operator at $p$. To do so, we assume that $R \subset \CC$ is the localization of a finite $S^0$-algebra (or $S^0[\tau]$-algebra when considering a fixed $P$-nebentypus $\tau$) at the maximal ideal determined by $\incl_p$ or that $\iota_p(R)$ is $p$-adically complete (in the latter case, we say that $R$ is a \emph{$p$-adic algebra}).

If $R$ is a $p$-adic algebra, the $P$-ordinary projector $e_{P, \kappa} = e_\kappa$ defined in Section \ref{Pord rep defn} has a well-defined action on $\bTT_{K_r, \kappa, R}$ and we set $\bTT^{\Pord}_{K_r, \kappa, R} := e_\kappa\bTT_{K_r, \kappa, R}$. Similarly, let $\bTT^{\Pord}_{K_r, \kappa, \tau, R} := e_\kappa\bTT_{K_r, \kappa, \tau, R}$. These are equal to the quotient algebras obtained from $\bTT_{K_r, \kappa, R}$ and $\bTT_{K_r, \kappa, \tau, R}$ upon restriction of the operators to the (stable) subspaces $S_\kappa^{\Pord}(K_r; R) := e_\kappa S_\kappa(K_r; R)$ and $S_\kappa^{\Pord}(K_r, \tau; R) := e_\kappa S_\kappa(K_r, \tau; R)$.

Similarly, when $R$ is not $p$-adic, we can define the latter spaces as the intersection of $S_\kappa(K_r; R)$ and $S_\kappa(K_r, \tau; R)$ with the $P$-ordinary spaces over the ($p$-adic) completion of $\incl_p(R)$. 

Assume the holomorphic representation $\pi$ from Section \ref{hecke alg on cusp forms} is $P$-ordinary at $p$. Assume that $\pi_p$ satisfies the Hypothesis \ref{Qw equals Pw} on supercuspidal support as in Section \ref{BK types of Pord rep}. Therefore it has a well-defined BK-type $\tau$. Let $\pi_p^{(P, \tau)}$ be the subspace of $P$-ordinary vectors in $\pi_p^{I_r}$ of type $\tau$, as in Theorem \ref{Pord structure thm}.

The Hecke algebra $\bTT_{K_r, \kappa, R}$ acts on $\pi_p^{(P, \tau)} \otimes \pi^{p, K^p}$ via a character $\lambda_\pi$ that extends $\lambda_\pi^p$. Clearly, the character $\lambda_\pi$ factors through $\bTT^{\Pord}_{K_r, \kappa, \tau, R}$. Let $E(\lambda_\pi)$ be the finite extension of $E(\pi)$ generated by the values of $\lambda_\pi$ and let $R(\lambda_\pi)$ be the localization of the ring of integers of $E(\lambda_\pi)$ at the maximal ideal determined by $\incl_p$. One readily sees that $\lambda_\pi$ is $R(\lambda_\pi)$-valued.

Let $\ol{\lambda}_\pi$ be the reduction of $\lambda_\pi$ modulo the maximal ideal of $R(\lambda_\pi)$, viewed as a character valued in $\ol{\ZZ}_{(p)}$. Denote its kernel by $\m_\pi$ and let
\begin{align*}
    \SSS(K_r, \kappa, \tau, \pi) = 
    \left\{
        \substack{
            \text{ordinary holomorphic cuspidal $\pi'$ of weight type $(\kappa, K_r)$,} \\
            \text{satisfying Hypothesis \ref{Qw equals Pw} with BK type $\tau$, such that $\ol{\lambda}_\pi = \ol{\lambda}_{\pi'}$}
        }
    \right\}
\end{align*}

Clearly, the condition $\ol{\lambda}_\pi = \ol{\lambda}_{\pi'}$ is equivalent to $\m_\pi = \m_{\pi'}$.

\begin{lemma} \label{Pord forms as lattices}
    Let $\pi$ be a holomorphic $P$-ordinary cuspidal automorphic form of weight type $(\kappa, K_r)$ and BK-type $\tau$ as above. Suppose that $\pi$ satisfies the multiplicity one Hypothesis \ref{mult one hyp}. Let $R \subset \CC$ be the localization of a finite extension of $R(\lambda_\pi)$ or the $p$-adic completion of such a ring. Let $E = R[\frac{1}{p}]$.
    \begin{enumerate}
        \item[(i)] Let $S_\kappa^{\Pord}(K_r, \tau; E)[\lambda_\pi] = e_\kappa S_\kappa(K_r, \tau; E)[\lambda_\pi]$, where $[\lambda_\pi]$ denotes $\lambda_\pi$-isotypic component. Then, $j_\pi$ restricts to an isomorphim
        \[
            j_\pi : 
                \pi_p^{(P, \tau)} \otimes \pi_S^{K_S}
            \xrightarrow{\sim}
                S_\kappa^{\Pord}(K_r, \tau; E)[\lambda_\pi] \otimes_E \CC
        \]
        
        \item[(ii)] Let $S_\kappa^{\Pord}(K_r, \tau; R)_{\pi}$ be the localization of $S_\kappa^{\Pord}(K_r, \tau; R)$ at $\m_\pi$ and
        \[
            S_\kappa^{\Pord}(K_r, \tau; R)[\pi] 
            := 
                S_\kappa^{\Pord}(K_r, \tau; R)_{\pi}
                    \cap
                S_\kappa^{\Pord}(K_r, \tau; E)[\lambda_\pi]
        \]

        Then, $j_\pi$ identifies $S_\kappa^{\Pord}(K_r, \tau; R)[\pi]$ as an $R$-lattice in $\pi_p^{(P, \tau)} \otimes \pi_S^{K_S}$. Similarly, $S_\kappa^{\Pord}(K_r, \tau; R)_{\pi}$ is identified with an $R$-lattice in
        \[
            \bigoplus \limits_{\pi' \in \SSS(K_r, \kappa, \tau, \pi)} (\pi'_p)^{(P, \tau)} \otimes (\pi'_S)^{K_S}
        \]
        via $\oplus_{\pi'} j_{\pi'}$.
    \end{enumerate}
\end{lemma}

\subsection{$P$-anti-ordinary case.} \label{Hecke alg for Paord}
In this section, we carry a similar analysis as in the previous section for anti-holomorphic and $P$-anti-ordinary representations. 

Assume that $R \subset \CC$ is as in the beginning of Section \ref{Hecke alg for Pord}. Then, one can define $H_\kappa^{d, \Paord}(K_r; R)$ and $H_\kappa^{d, \Paord}(K_r, \tau; R)$ by replacing $S_\kappa(K_r; R)$ and $S_\kappa(K_r, \tau; R)$ with $H_\kappa^{d}(K_r; R)$ and $H_\kappa^{d}(K_r, \tau; R)$ respectively, and $e_\kappa$ by $e_\kappa^-$ (see Section \ref{a.holo and a.ord on G1}). Restriction of the operators in $\bTT^d_{K_r, \kappa, R}$ to these \emph{$P$-anti-ordinary subspaces} yields $\bTT^{d, \Paord}_{K_r, \kappa, R} := e_\kappa \bTT^{d}_{K_r, \kappa, R}$ and $\bTT^{d, \Paord}_{K_r, \kappa, \tau, R} := e_\kappa \bTT^{d}_{K_r, \kappa, \tau, R}$ as quotient $R$-algebras. The following is obvious from the definitions.
\begin{lemma} \label{iso Pord hecke alg and Paord hecke alg}
    For $R$ as above, the isomorphisms of Lemma \ref{iso holo hecke alg and aholo hecke alg} induce $R$-algebra isomorphisms
    \[
        \bTT^{\Pord}_{K_r, \kappa, R}
            \xrightarrow{\sim}
        \bTT^{d, \Paord}_{K_r, \kappa^D, R}
        \ \ \ \text{and} \ \ \
        \bTT^{\Pord}_{K_r, \kappa, \tau, R}
            \xrightarrow{\sim}
        \bTT^{d, \Paord}_{K_r, \kappa^D, \tau^\vee, R}
    \]
\end{lemma}

Proceeding as in the previous sections, let $\pi$ be a holomorphic cuspidal automorphic representation of weight type $(\kappa, K_r)$. Then, $\pi^\flat$ is anti-holomorphic of type $(\kappa, K_r)$. 

Again, the choice of an $E(\pi)$-rational spherical vector in $\pi_f^{\flat, S}$ and a choice of basis for the $1$-dimensional $H^d(\P_h, K_h; \pi_\infty^\flat \otimes W_{\kappa^D})$ induces an inclusion
\[
    j_{\pi^\flat}^\vee : 
        \pi_S^{\flat, K_S} 
            \otimes 
        \pi_p^{\flat, I_r} 
    \hookrightarrow 
        H^d_{\kappa^d}(K_r; \CC)
\]

Furthermore, if we assume that $\pi$ is $P$-ordinary and satisfies Hypothesis \ref{Qw equals Pw}, hence has a well-defined BK-type $\tau$. In that case, $\pi^\flat$ is $P$-anti-ordinary and determines a character $\lambda_{\pi^\flat}$ of $\bTT_{K_r, \kappa^D, \tau^\vee}^{d, \Paord}$ as well as a maximal ideal $\m_{\pi^\flat}$. Via the isomorphism 
$
    \bTT^{\Pord}_{K_r, \kappa, \tau, R}
        \xrightarrow{\sim}
    \bTT^{d, \Paord}_{K_r, \kappa^D, \tau^\vee, R} \ ,
$
we have $\lambda_{\pi^\flat} = \lambda_{\pi}$ and $\m_\pi = \m_{\pi^\flat}$. Let $\pi_{p, r}^{\flat, (P, \tau^\vee)}$ denote the subspace of $\pi_p^{\flat, I_r}$ of $P$-anti-ordinary vector of level $r$ and type $\tau^\vee$ (see Corollary \ref{Paord structure thm}).

\begin{lemma} \label{Paord forms as lattices}
    Let $\pi$, $\kappa$, $K_r$, $\tau$, $R$ and $E$ be as in Lemma \ref{Pord forms as lattices}. Let $H_{\kappa^D}^{d, \Paord}(K_r, \tau^\vee; R)_{\pi}$ be the localization of $H_{\kappa^D}^{d, \Paord}(K_r, \tau^\vee; R)$ at $\m_{\pi}$ and let
        \[
            H_{\kappa^D}^{d, \Paord}(K_r, \tau^\vee; R)[\pi] 
            := 
                H_{\kappa^D}^{d, \Paord}(K_r, \tau^\vee; R)_{\pi}
                    \cap
                H_{\kappa^D}^{d, \Paord}(K_r, \tau^\vee; E)[\lambda_\pi] \ ,
        \]
    where $[\lambda_pi]$ denotes $\lambda_\pi$-isotypic subspace again. Then, if $\pi$ satisfies the multiplicity one Hypothesis \ref{mult one hyp},
    \begin{enumerate}
        \item[(i)] The inclusion $j_{\pi^\flat}^\vee$ restricts to an isomorphism
        \[
            \pi_S^{\flat, K_S} 
                \otimes 
            \pi_{p, r}^{\flat, (P, \tau^\vee)} 
                \hookrightarrow 
            H^d_{\kappa^d, \Paord}(K_r, \tau; E)[\lambda_\pi] \otimes_E \CC
        \]
        
        \item[(ii)] The map $j_{\pi^\flat}^\vee$ identifies $H_{\kappa^D}^{d, \Paord}(K_r, \tau^\vee; R)[\pi]$ with an $R$-lattice in 
        $
            \pi_S^{\flat, K_S} 
                \otimes 
            \pi_{p, r}^{\flat, (P, \tau^\vee)}
        $. Furthermore, $H_{\kappa^D}^{d, \Paord}(K_r, \tau^\vee; R)_{\pi}$ is identified with an $R$-lattice in
        \[
            \bigoplus_{\pi' \in \SSS(K_r, \kappa, \tau, \pi)}
                (\pi'_S)^{\flat, K_S} 
                    \otimes 
                (\pi'_{p, r})^{\flat, (P, \tau^\vee)}
        \]
        via $\oplus_{\pi'} j_{(\pi')^\flat}^\vee$.

        \item[(iii)] The normalized Serre duality pairing induces a perfect $\bTT^{\Pord}_{K_r, \kappa, \tau, R}$-equivariant pairings
        \begin{align*}
            S_{\kappa}^{\Pord}(K_r, \tau; R)[\pi] 
                \otimes_R 
            H_{\kappa^D}^{d, \Paord}(K_r, \tau^\vee; R)[\pi]
                \to 
            R \ &\text{and} \\
            S_{\kappa}^{\Pord}(K_r, \tau; R)_{\pi} 
                \otimes_R 
            H_{\kappa^D}^{d, \Paord}(K_r, \tau^\vee; R)_{\pi}
                \to 
            &R \ .
        \end{align*}
    \end{enumerate}
\end{lemma}

\bibliography{references.bib}

\newcommand{\etalchar}[1]{$^{#1}$}
\providecommand{\bysame}{\leavevmode\hbox to3em{\hrulefill}\thinspace}
\providecommand{\MR}{\relax\ifhmode\unskip\space\fi MR }
\providecommand{\MRhref}[2]{%
  \href{http://www.ams.org/mathscinet-getitem?mr=#1}{#2}
}
\providecommand{\href}[2]{#2}
\begin{thebibliography}{CEF{\etalchar{+}}16}

\bibitem[BHR94]{BHR94}
D.~Blasius, M.~Harris, and D.~Ramakrishnan, \emph{Coherent cohomology, limits
  of discrete series, and {G}alois conjugation}, Duke Math. J. \textbf{73}
  (1994), no.~3, 647–685.

\bibitem[BK98]{BusKut98}
C.~J. Bushnell and P.~C. Kutzko, \emph{Smooth representations of reductive
  \textit{p}-adic groups : {S}tructure theory via types}, Proceedings of the
  London Mathematical Society \textbf{77} (1998), no.~3, 582--634.

\bibitem[BK99]{BusKut99}
\bysame, \emph{Semisimple types in {GL}(\textit{n})}, Compositio Mathematica
  \textbf{119} (1999), 57--106.

\bibitem[Cas95]{Cas95}
W~Casselman, \emph{Introduction to the {T}heory of {A}dmissible
  {R}epresentations of \textit{p}-adic {R}eductive {G}roups}, Unpublished
  manuscript.
  \url{https://personal.math.ubc.ca/~cass/research/pdf/p-adic-book.pdf}, 1995,
  80 pages.

\bibitem[CEF{\etalchar{+}}16]{CEFMV}
A.~Caraiani, E.~Eischen, J.~Fintzen, E.~Mantovan, and I.~Varma, \emph{p-adic
  q-expansion {P}rinciples on {U}nitary {S}himura {V}arieties}, Directions in
  Number Theory, Association for Women in Mathematics Series, vol.~3, Springer
  International Publishing, Cham, 2016, pp.~197--243.

\bibitem[EHLS20]{EHLS}
E.~Eischen, M.~Harris, J.-S. Li, and C.~Skinner, \emph{p-adic {L}-functions for
  unitary groups}, Forum of Mathematics, Pi \textbf{8} (2020), e9.

\bibitem[Eis12]{Eis12}
E.~Eischen, \emph{\textit{p}-adic differential operators on automorphic forms
  on unitary groups}, Ann. Inst. Fourier (Grenoble) \textbf{62} (2012), no.~1,
  177--243.

\bibitem[Eis15]{Eis15}
\bysame, \emph{A \textit{p}-adic {E}isenstein measure for unitary groups}, J.
  Reine Angew. Math. \textbf{699} (2015), 111--142.

\bibitem[GPSR87]{GPSR87}
S.~Gelbart, I.~Piatetski-Shapiro, and S.~Rallis, \emph{Explicit {C}onstructions
  of {A}utomorphic {L}-functions}, Lecture Notes in Mathematics, vol. 1254,
  Springer, Berlin, 1987.

\bibitem[Har86]{Har86}
M.~Harris, \emph{Arithmetic vector bundles and automorphic forms on {S}himura
  varieties. {II}}, Compositio Math. \textbf{60} (1986), no.~3, 323--378.

\bibitem[Hid98]{Hid98}
H.~Hida, \emph{Automorphic induction and {L}eopoldt type conjectures for
  {GL}(\textit{n})}, Asian J. Math. \textbf{2} (1998), no.~4, 667--710, Mikio
  Sato: a great Japanese mathematician of the twentieth century.

\bibitem[Hid04]{Hid04}
\bysame, \emph{p-adic {A}utomorphic {F}orms on {S}himura {V}arieties}, Springer
  Monographs in Mathematics, Springer, New York, NY, 2004.

\bibitem[HLS06]{HLS}
M.~Harris, J.S. Li, and C.~Skinner, \emph{\textit{p}-adic \textit{L}-functions
  for unitary {S}himura varieties, {I}: {C}onstruction of the {E}isenstein
  {M}easure}, Doc. Math. \textbf{\textbf{Extra Vol.}} (2006), 393--464,
  (electronic).

\bibitem[Jan03]{Jan03}
J.~C. Jantzen, \emph{Representations of {A}lgebraic {G}roups}, second ed.,
  Mathematical surveys and monographs, vol. 107, American Mathematical Society,
  Providence, RI, 2003.

\bibitem[Kat78]{Kat78}
N.~M. Katz, \emph{\textit{p}-adic \textit{L}-functions for {CM} fields},
  Invent. Math. \textbf{49} (1978), no.~3, 199--297.

\bibitem[Kot92]{Kot92}
R.~E. Kottwitz, \emph{Points on {S}ome {S}himura {V}arieties {O}ver {F}inite
  {F}ields}, J. Amer. Math. Soc. \textbf{5} (1992), no.~2, 373--444.

\bibitem[Lan12]{Lan12}
K.-W. Lan, \emph{Comparison between analytic and algebraic constructions of
  toroidal compactifications of {PEL}-type {S}himura varieties}, Journal für
  die reine und angewandte Mathematik (Crelles Journal) \textbf{2012} (2012),
  no.~664, 163--228.

\bibitem[Lan13]{Lan13}
\bysame, \emph{Arithmetic {C}ompactifications of {PEL}-type {S}himura
  {V}arieties}, London Mathematical Society Monographs, vol.~36, Princeton
  University Press, Princeton, 2013.

\bibitem[Lat21]{Lat21}
P.~Latham, \emph{Typical representations, parabolic induction and the inertial
  local {L}anglands correspondence}, 2021, Preprint available at
  \href{https://arxiv.org/abs/2101.04900}{arXiv:2101.04900}.

\bibitem[LR20]{LiuRos20}
Z.~Liu and G.~Rosso, \emph{Non-cuspidal {H}ida theory for {S}iegel modular
  forms and trivial zeros of \textit{p}-adic \textit{L}-functions}, Math. Ann.
  \textbf{378} (2020), 153--231.

\bibitem[Mar23]{Mar23c}
D.~Marcil, \emph{\textit{p}-adic \textit{L}-functions for \textit{P}-ordinary
  {H}ida families on unitary groups.}, 2023, In progress.

\bibitem[Pas05]{Pas05}
V.~Paskunas, \emph{Unicity of {T}ypes for {S}upercuspidal {R}epresentations of
  {GL}$_{N}$}, Proceedings of the London Mathematical Society \textbf{91}
  (2005), no.~3, 623--654.

\bibitem[Pil12]{Pil12}
V.~Pilloni, \emph{Sur la théorie de {H}ida pour le groupe {GSp}$_{2g}$}, Bull.
  Soc. Math. France \textbf{140} (2012), no.~3, 335--400.

\bibitem[Ren10]{Ren10}
D.~Renard, \emph{Repr{é}sentations des groupes r{é}ductifs p-adiques},
  Collection SMF. Cours sp{é}cialis{é}s, vol.~17, Soci{é}t{é}
  {M}ath{é}matique de France, Paris, France, 2010.

\bibitem[SU02]{SU02}
D.~Skinner and E.~Urban, \emph{Sur les déformations \textit{p}-adiques des
  formes de {S}aito-{K}urokawa}, C. R. Math. Acad. Sci. Paris \textbf{335}
  (2002), no.~7, 581--586.

\end{thebibliography}
\bibliographystyle{amsalpha}

\end{document}